\documentclass[12pt,english]{article}
\usepackage[a4paper]{geometry}
\usepackage{color}
\usepackage{mathtools}
\usepackage{bm}
\usepackage{amsmath}
\usepackage{amssymb,amsthm}
\usepackage{esint}

\makeatletter

\usepackage{bbm}
\usepackage{dsfont}
\usepackage{changepage}
\usepackage[all]{xy}
\usepackage{amscd}
\usepackage{geometry}
\usepackage[bookmarks]{hyperref}
\usepackage{babel}
\newcommand{\X}{\mathbb{X}}
\newcommand{\mrestr}{\mathbin{\vrule height 1.6ex depth 0pt width
0.13ex\vrule height 0.13ex depth 0pt width 1.3ex}}
\newcommand{\mean}[1]{\,-\hskip-1.08em\int_{#1}}

\def\R{\mathbb{R}}
\def\N{\mathbb{N}}
\def\one{\mathds1}
\def\d{\mathrm{d}}
\def\dd{\bm{\delta}}
\def\lip{\mathrm{Lip}}
\def\der{\mathbf{Der}}
\def\loc{\mathrm{loc}}
\def\M{\mathbf{M}}
\renewcommand{\div}{\text{\sl div}}
\newtheorem{definition}{Definition}[section]
\newtheorem{thm}{Theorem}[section]
\newtheorem{proposition}[thm]{Proposition}

\newtheorem{r1}[thm]{Remark}
\newenvironment{remark}{\begin{r1}
\begin{rm}}{
\end{rm}\end{r1}}
\newtheorem{lemma}[thm]{Lemma}

\newcommand{\subjclass}[2][1991]{
  \let\@oldtitle\@title
  \gdef\@title{\@oldtitle\footnotetext{#1 \emph{Mathematics subject classification.} #2}}
}
\newcommand{\keywords}[1]{
  \let\@@oldtitle\@title
  \gdef\@title{\@@oldtitle\footnotetext{\emph{Key words and phrases.} #1.}}
}

\title{Rough traces of $BV$ functions in metric measure spaces}
\author{Vito Buffa\footnote{Bologna, Italy. E--mail: \texttt{bff.vti@gmail.com}, \textsf{ORCID iD}: 0000-0003-4175-4848.},
Michele Miranda jr.\footnote{University of Ferrara, 
Ferrara, Italy. E--mail: \texttt{michele.miranda@unife.it}}}
\date{\today}
\subjclass[2020]{Primary 26A45. Secondary 26B20, 30L99, 46E35, 46E36.}
\keywords{Functions of bounded variation, metric measure spaces, traces, integration by parts formulas}

\makeatother

\allowdisplaybreaks

\begin{document}

\maketitle

\begin{abstract}\noindent
Following a Maz'ya-type approach,
we adapt the theory of rough traces of functions of bounded variation ($BV$) to the context of doubling metric measure spaces supporting a Poincar\'e inequality. This eventually allows for an integration by parts formula involving the rough trace of such functions. We then compare our analysis
with the study done in a recent work by P. Lahti and N. Shanmugalingam, where traces of $BV$ functions are studied by means of the more classical Lebesgue-point characterization,
and we determine the conditions under which the two notions coincide.
\end{abstract}

\bigskip
\sloppy

\section{Introduction}

\medskip

This paper aims at investigating traces of $BV$ functions and integration 
by parts formul\ae\ in metric measure spaces. The setting is given by a complete and separable metric measure  space $(\X,d,\mu)$ endowed with a doubling measure $\mu$ and supporting a weak $(1,1)$--Poincar\'e inequality. We prove an integration by parts formula on open sets of finite perimeter with some regularity; the basic idea of the proof is to use the notion of essential boundary and to define the rough trace of a $BV$ function on such boundary using its super-level sets.

Sets with finite perimeter in metric measure spaces were defined for instance in \cite{Mir} and studied by L. Ambrosio in \cite{A,A1}. The main fact we use is that, for a set with finite perimeter, the perimeter measure is concentrated on the essential boundary of the set itself, \cite{A1}. The notion of essential boundary is good enough to perform the strategy given by V. Maz'ya in his book~\cite{Ma}. In the Euclidean case, the reduced boundary was used instead, and an integration by parts formula was proved. Also, the continuity of the trace operator was investigated and equivalent conditions for such continuity were given. In the metric space setting - except for the case of $\mathsf{RCD}(K,N)$ spaces \cite{abs,bps} - we have so far no good notion of reduced boundary, but for our aims the essential boundary suffices. 

Properties of the trace operator have been recently investigated in \cite{LS} and sufficient conditions for the continuity of such operator were given in terms of a ``measure-density condition'' on the boundary of the selected domain. We compare this notion of trace with the rough trace proving almost-everywhere equality of the two functions on the boundary of the open set under investigation. In this way, two different characterizations of the trace values of a function with bounded variation are available, the two being equivalent. 

\bigskip

The paper is organized as follows.

\smallskip

In Section \ref{preliminaries} we review the basic tools of our analysis, namely the concept of a metric measure space $(\X,d,\mu)$ equipped with a doubling measure and supporting a weak Poincar\'e inequality, the notions of $BV$ function and of Caccioppoli set, along with the fundamental results related to them, such as the Coarea Formula, the Isoperimetric Inequality and of course the remarkable Theorem by L. Ambrosio on the Hausdorff representation of the perimeter measure, \cite[Theorem 5.3]{A1}.

\smallskip

In Section \ref{roughtrace} we rewrite, after \cite{Ma}, the notion and the properties of the rough trace of $BV$ functions defined on an open domain $\Omega\subset\X$. In particular, we re-investigate the conditions under which a $BV$ function admits a summable rough trace and we consider the issue of the extendability of $u\in BV(\Omega)$ to the whole of $\X$. The latter part of the Section is then devoted to an integration by parts formula for functions of bounded variation in terms of a suitable class of ``vector fields'' (actually, \textit{bounded Lipschitz derivations}), a formula which, as shown in Theorem \ref{gauss-green}, features implicitly the rough trace of $u\in BV(\Omega)$. 

The topic of integration by parts formul\ae, especially in connection with $BV$ functions and sets of finite perimeter, has been an object of interest for quite a few decades now. After the pioneering work of G. Anzellotti \cite{anz} in 1983, who introduced the class of \textit{divergence-measure vector fields} - namely, those vector fields whose distributional divergence is a finite Radon measure - to prove an integration by parts formula for $BV$ functions on domains with Lipschitz boundary, such research area has been flourishing again since the early 2000's, when several authors started devoting considerable attention to the subject, leading to notable applications to sets of finite perimeter in the Euclidean setting, namely, the validity of (generalized) Gauss-Green formul\ae\ in terms of the normal traces of divergence-measure fields (see \cite{acm,ctz,sil,sil2}, and also the latest developments given in \cite{cct,cp,ct,cdc1,cdc2,cdcm,leosar1,leosar2}). \\
More recently, the issue has been attacked also in less regular settings, like metric measure spaces, \cite{bps,bu,bcm,mms}, and stratified groups, \cite{cm}. In particular, in \cite{mms} the authors operated in the context of a doubling metric measure space equipped with Cheeger's differential structure \cite{ch} and satisfying a Poincar\'e inequality; in their analysis, they found the so-called \textit{regular balls} to be the appropriate class of domains where a certain integration by parts formula holds. The results of \cite{mms} were then reprised by \cite{bu} and later refined in \cite{bcm}; both these works rely on the differential structure developed in \cite{gi}, which allows to extend the previous analysis of \cite{mms} to the very abstract context of a metric measure space satisfying no specific structural assumptions, where \textit{regular domains}\footnote{See Remark \ref{rmk-gauss} for the definition of regular domains.} serve as a generalization of regular balls. In particular, \cite{bcm} specializes the discussion for $\mathsf{RCD}(K,\infty)$ spaces and $BV$ functions. Lastly, in the more recent paper \cite{bps}, a Gauss-Green formula for sets of finite perimeter was proved in the context of an $\mathsf{RCD}(K,N)$ space by means of Sobolev vector fields in the sense of \cite{gi}.

\smallskip

In Section \ref{comparison}, finally, we compare our approach with the results recently obtained in \cite{LS} about the trace operator for $BV$ functions defined by means of Lebesgue points. Our analysis eventually allows to find the optimal conditions to impose on the domain $\Omega$ in order to ensure the coincidence in the $\mathcal{S}^h$-almost everywhere sense - Theorem \ref{u-star-aplim} -  between the rough and the ``classical'' traces, $u^*(x)=\text{T}u(x)$.

\smallskip

Sections \ref{roughtrace} and \ref{comparison} extend and refine the results contained in \cite[Section 7.2]{bu}.

\bigskip

\section{Preliminaries}\label{preliminaries}

\medskip

Throughout this paper, $(\X,d,\mu)$ will be a complete and separable metric measure space equipped with a non-atomic, non-negative Borel measure $\mu$ such that $0<\mu\left(B_\rho(x)\right)<\infty$ for any ball $B_\rho(x)\subset\X$ with radius $\rho>0$ centered at $x\in\X$. By \textit{non-atomic} we mean that for every $x\in\X$ one has $\mu(\{x\})=0$.

We shall assume $\mu$ to be \textit{doubling}:
in other words, there exists a constant $c\ge1$ such that
\begin{equation}\label{doubling}
\mu(B_{2\rho}(x))\leq c \mu\left(B_\rho(x)\right),
\qquad \forall\, x\in \X, \forall \,\rho>0.
\end{equation}

The minimal constant appearing in \eqref{doubling} is called \textit{doubling constant} and will be denoted by $c_D$; $s\coloneqq\ln_2 c_D$ is the \textit{homogeneous 
dimension} of the metric space $\X$
and it is known that the following
property holds:

\begin{equation}\label{hom-dim}
\frac{\mu(B_r(x))}{\mu(B_R(y))}
\geq \frac{1}{c^2_D} \left(
\frac{r}{R}\right)^s,
\end{equation}

for every $y\in\X$, $x\in B_R(y)$, and for every $0<r\le R<\infty$ (see for example \cite[Lemma 3.3]{bjorn}).

The Lebesgue spaces $L^p(\X,\mu)$, $1\leq
p\leq \infty$ are defined in the usual way, \cite{hkst}; 
since in a complete doubling metric measure space
balls are totally bounded, we can equivalently 
set $L^p_{\rm loc}(\X,\mu)$ to denote
the space of functions that belong to 
$L^p(K,\mu)$ for any compact set $K$ or that
belong to $L^p(B_\rho(x_0),\mu)$ for any
$x_0\in \X$ and any $\rho>0$. 

Given a Lipschitz function $f:\X\to \R$, we define the \textit{pointwise Lipschitz constant} of $f$ as
\begin{equation*}
{\rm Lip}(f)(x)\coloneqq\limsup_{y\to x}
 \frac{|f(x)-f(y)|}{d(x,y)}.
\end{equation*}
We assume that the space supports
a weak $(1,1)$--Poincar\'e inequality,
which means that there exist constants 
$c_P>0, \lambda\geq1$ such that for any Lipschitz function $f$
\begin{equation}\label{poincare}
\int_{B_\rho(x)} |f(y)-f_{B_\rho(x)}| 
\mathrm{d}\mu(y) 
\leq 
c_P \rho \int_{B_{\lambda \rho}(x)}
|{\rm Lip}(f)(y)| \mathrm{d}\mu(y),
\end{equation}
where $f_E$ is the mean value
of $f$ over the set $E$, i.e. if
$\mu(E)\neq 0$
\[
f_E \coloneqq\frac{1}{\mu(E)} \int_E f(y)\mathrm{d}\mu(y).
\]
We recall also the definition 
of upper gradient; we say that
a Borel function $g:\X\to [0,+\infty]$
is an \textit{upper gradient} for a measurable
function $f$ if for any rectifiable Lipschitz
curve 
$\gamma:[0,1]\to \X$ with endpoints $x,y\in\X$ we have that
\[
|f(x)-f(y)|\leq \int_\gamma g \coloneqq
\int_0^1 g(\gamma(t)) \|\gamma'(t)\|
\mathrm{d}t
\]
where $\|\gamma'(t)\|={\rm Lip}(\gamma)(t)$.


In what follows, we shall also need to quantify how ``dense'' is a set at a certain point of the space; then, the \textit{upper and lower $\mu$-densities}
of $E\subset\X$ at $x\in \X$ are given by
\[
\Theta^{*}_\mu(E,x)\coloneqq\limsup_{\rho\to 0}
\frac{\mu(E\cap B_\rho(x))}{\mu(B_\rho(x))},
\]
and
\[
\Theta_{*,\mu}(E,x)\coloneqq\liminf_{\rho\to 0}
\frac{\mu(E\cap B_\rho(x))}{\mu(B_\rho(x))}
\]
respectively. The common value between the two limits will be called the \textit{$\mu$-density} of $E$ at $x\in\X$, denoted by
\[
 \Theta_\mu(E,x)\coloneqq\underset{\rho\to 0}{\lim}\frac{\mu\left(E\cap B_\rho(x)\right)}{\mu\left(B_\rho(x)\right)}.
\]
When we work with the reference measure $\mu$ only and there is no ambiguity, we shall drop the suffix from the notation and the above will be simply referred to as the (upper, lower) \textit{density} of $E$ at $x$.

The left continuity of the maps $\rho\mapsto \mu(E\cap B_\rho(x))$
for any Borel set $E$ implies that the maps $x\mapsto \mu(E\cap B_\rho(x))$ are lower semicontinuous.
From this, one deduces that functions $\Theta^*(E,x)$ and 
$\Theta_*(E,x)$ are Borel. 

Following the characterization given for instance in \cite[Definition 3.60]{afp}, for a Borel set $E\subset \X$ we shall denote by $E^{(t)}$, $t\in [0,1]$, the set of points where $E$ has density $t$, namely
\[
 E^{(t)}\coloneqq\left\{x\in\X:\;\Theta(E,x)=\underset{\rho\to 0}{\lim}\frac{\mu\left(B_\rho (x)\cap E\right)}{\mu\left(B_\rho(x)\right)}=t\right\}.
\]
In particular, the sets $E^{(0)}$ and $E^{(1)}$ will be called the \textit{measure-theoretic} (or, \textit{essential}) \textit{exterior} and \textit{interior} of $E$, respectively.

The \textit{measure-theoretic} (or, \textit{essential}) \textit{boundary} of $E$ is then defined as

\begin{equation}\label{dstar}
\partial^*E \coloneqq \X
\backslash(E^{(0)}\cup E^{(1)}).
\end{equation}

Note that we could equivalently characterize $\partial^*E$ as the set of points $x\in\X$ where both $E$ and its complement $E^c$ have positive upper density.

The \textit{lower and upper approximate limits} of any measurable function $u:\X\to\mathbb{R}$ at $x\in\X$ are defined by
\begin{equation}\label{low-aplim}
u^{\wedge}(x)\coloneqq\sup\left\{t\in\overline{\mathbb{R}}:\;\underset{\rho\to 0}{\lim}\frac{\mu\left(B_\rho(x)\cap E_t\right)}{\mu\left(B_\rho(x)\right)}=1\right\}
\end{equation}
and
\begin{equation}\label{up-aplim}
 u^{\vee}(x)\coloneqq\inf\left\{t\in\overline{\mathbb{R}}:\;\underset{\rho\to 0}{\lim}\frac{\mu\left(B_\rho(x)\cap E_t\right)}{\mu\left(B_\rho(x)\right)}=0\right\}
\end{equation}
respectively, where for $t\in\mathbb{R}$, $E_t$ denotes the \textit{super-level sets} of the function $u$, namely
\[
 E_t\coloneqq\left\{x:\;u(x)\ge t\right\}.
\]
We observe that the density condition in \eqref{low-aplim}  is of course equivalent to ask that $\Theta(E_t^c,x)=0$.
The notion of approximate limits allows for the characterization of a \textit{jump set} of the function $u$:
\begin{equation}\label{jump}
 S_u\coloneqq\left\{x\in\X:\;u^{\wedge}(x)<u^{\vee}(x)\right\}.
\end{equation}
So in particular, when $u=\one_E$, one gets $S_u=\partial^*E$.

We also notice that, if $u$ is bounded above and 
$t>\text{ess\,-}\sup u$, then $\mu(E_t)=0$, hence
$u^\vee(x)\leq\text{ess\,-}\sup u$. In the same way, if
$u$ is bounded below, $u^\wedge(x)\geq\text{ess\,-}\inf u$.

\smallskip

Following \cite{A1,amp,Mir}, we now briefly recall the basic notions and properties of functions of bounded variation on metric measure spaces.
Given an open set $\Omega\subset\X$, we
define the \textit{total variation} of a measurable function 
$u:\Omega\to \R$ by setting
\[
\|Du\|(\Omega)
=\underset{\mathcal{A}_u}{\inf}\left\{
\liminf_{j\to +\infty} 
\int_{\Omega} {\rm Lip}(u_j)(y) \mathrm{d}\mu(y)
\right\},
\]
where
\[
 \mathcal{A}_u\coloneqq\left\{\left(u_j\right)_{j\in\mathbb{N}}\subset\mathrm{Lip}_{\mathrm{loc}}(\Omega):\;u_j\to u\:\mathrm{in}\:L^1_{\mathrm{loc}}(\Omega)\right\}.
\]

\begin{definition}
Given $u\in L^1(\Omega)$, we say that
$u$ has bounded variation in $\Omega$, $u\in BV(\Omega)$,
if $\|Du\|(\Omega)<+\infty$. A set $E\subset \X$
is said to have finite perimeter in $\X$ if $\one_E\in 
BV(\X)$, and similarly to have finite perimeter in $\Omega$ if $\one_E\in BV(\Omega)$.
\end{definition}

Sets of finite perimeter will be also referred to as \textit{Caccioppoli sets}.

A function $u\in BV(\Omega)$ defines a non-negative
Radon measure $\|Du\|$, the \textit{total variation measure}; when $u$ is the characteristic
function of some set $E$, $u=\one_E$, then $\|D\one_E\|$ is called \textit{perimeter measure}. 

A very important tool for our work will be the \textit{Coarea Formula}, \cite{Mir}, which asserts that for $u\in BV(\Omega)$, then for almost every $t\in \R$ the set $E_t$ has finite perimeter in $\Omega$ and for any Borel set $A$
\[
\|Du\|(A)=\int_\R \|D\one_{E_t}\|(A)
\text{d}t.
\]
The Poincar\'e inequality and
the Sobolev embedding Theorem (see for 
instance \cite{A1}, \cite{HK} or \cite{Mir}) imply the following local isoperimetric inequality:
for any set $E$ with finite perimeter
and for any ball $B_\rho(x)$, we 
have that
\begin{align}\label{isoper}\begin{split}
\min\big\{\mu(E\cap B_\rho(x)), & \,\,\mu(B_\rho(x)\backslash E)\big\} \\ &\leq c_I \left(\frac{\rho^s}{\mu(B_\rho(x))}
\right)^{\frac{1}{s-1}}\|D\one_E\|\left(B_{\lambda\rho}(x)\right)^{\frac{s}{s-1}},\end{split}
\end{align}
where $c_I>0$ is known as the \textit{isoperimetric constant}.

We also mention that a weaker version of the Poincar\'e inequality holds for $BV$ functions as well: given any ball $B_\rho(x)\subset\X$, for every $u\in BV(\X)$ there holds
\begin{equation}\label{poincarebv}
 \int_{B_\rho(x)} \left|u-u_B\right|\text{d}\mu\le c_P\rho\|Du\|\!\left(B_{\lambda\rho}(x)\right).   
\end{equation}
Of course, both in \eqref{isoper} and \eqref{poincarebv} the notation is the same as in \eqref{poincare}.

\medskip

Two important properties of the perimeter measure of Caccioppoli sets, which we shall use extensively, are its absolute continuity with respect to the spherical Hausdorff measure and its localization inside the essential boundary, \cite{A1}.

Let us denote by ${\cal S}^h$ the spherical Hausdorff measure defined in terms of the doubling function
\[
h(B_\rho(x))=\frac{\mu(B_\rho(x))}{{\rm diam} (B_\rho(x))}.
\]
If $E\subset\X$ is a Caccioppoli set in $\X$, then we have the following result.


\begin{thm}\label{locality}{\normalfont{\cite[Theorem 5.3]{A1}}} The measure $\|D\one_E\|$ is concentrated on the set
\[
 \Sigma_\gamma = \left\{x:\;\underset{\rho\to 0}{\lim\sup}\min \left\{\frac{\mu\left(E\cap B_\rho(x)\right)}{\mu\left(B_\rho(x)\right)},\frac{\mu\left(E^c\cap B_\rho(x)\right)}{\mu\left(B_\rho(x)\right)}\right\}\ge\gamma\right\}\subset\partial^{*}E,
\]
where $\gamma=\gamma\left(c_D,c_I,\lambda\right)$. Moreover, $\mathcal{S}^h\left(\partial^{*}E\backslash\Sigma_\gamma\right)=0$, $\mathcal{S}^h\left(\partial^{*}E\right)<\infty$  and
\[
 \|D\one_E\|(B)=\int_{B\cap\partial^{*}E}\theta_E\mathrm{d}\mathcal{S}^h
\]
for any Borel set $B\subset\X$ and for some Borel map $\theta_E:\X\to[\alpha,\infty)$ with $\alpha=\alpha\left(c_P,c_I,\lambda\right)>0$. 
\end{thm}


\begin{remark}\label{per-bound}
 We explicitly observe that by \cite[Theorem 4.6]{amp} one actually has $\theta_E\in [\alpha,c_D]$; thus, with the same notation as in Theorem \ref{locality}, we are given the bounds
\begin{equation*}
 \alpha\mathcal{S}^h(B\cap\partial^* E)\le \|D\one_E\|(B)\le c_D\mathcal{S}^h(B\cap\partial^* E).
\end{equation*}
\end{remark}


\begin{definition}\label{local-space}{\normalfont{\cite{amp}}} The space $(\X,d,\mu)$ will be called \textit{local} if, given any two Caccioppoli sets $E,\Omega\subset\X$ with $E\subset\Omega$, one has that the maps arising from Theorem \ref{locality}, $\theta_E$ and $\theta_\Omega$, coincide $\mathcal{S}^h$-almost everywhere on $\partial^*\Omega\cap\partial^*E$. 
 \end{definition}


For instance, $\R^n$ equipped with the Euclidean distance and with the $n$-dimensional Lebesgue measure, is of course a local space; among other examples, we mention  non-collapsed $\mathsf{RCD}(K,N)$ spaces, Carnot groups of step 2 and certain classes of ``weighted'' spaces. See also \cite[Section 7]{amp} for a detailed discussion on possible examples of local spaces.


\begin{remark}\label{bv-decomp}
 In \cite[Theorem 5.3]{amp}, it was proved that for any function $u\in BV(\Omega)$, $\Omega\subset\X$ open set, the total variation measure $\|Du\|$ admits a decomposition into an ``absolutely continuous'' and a ``singular'' part, and that the latter is decomposable into a ``Cantor'' and a ``jump'' part. In other words, for any Borel set $B\subset\Omega$ the following holds:
 \begin{align*}
  \|Du\|(B) & = \|Du\|^{\text{a}}(B)+\|Du\|^{\text{s}}(B) \\
  & = \|Du\|^{\text{a}}(B)+\|Du\|^{\text{c}}(B)+\|Du\|^{\text{j}}(B) \\
  & = \int_{B} \mathfrak{a}\, \text{d}\mu + \|Du\|^{\text{c}}(B)+ \int_{B\cap S_u}\int_{u^\wedge(x)}^{u^\vee(x)}\theta_{\{u\ge t\}}(x)\text{d}t\text{d}\mathcal{S}^h(x),
 \end{align*}
where $\mathfrak{a}\in L^1(B)$ is the density of the absolutely continuous part, $\theta_{\{u\ge t\}}$ is given as in Theorem \ref{locality} and $S_u$ is the jump set as in \eqref{jump}.
\end{remark}


Another fact which we shall use is the following
localization property: if $E\subset \Omega$ has finite perimeter in an open set $\Omega$, then for all $x\in\Omega$ the function
\begin{equation*}\label{defmE}
m_E(x,\rho):=\mu(E\cap B_\rho(x))
\end{equation*}
is monotone non-decreasing as a function of $\rho$.  If it is differentiable at $\rho>0$, then
\begin{equation}\label{localization}
 \|D\one_{E\cap B_\rho(x)}\|(\Omega)\le m'_E(x,\rho)+\|D\one_E\|\left(\Omega\cap\overline{B}_\rho(x)\right).
\end{equation}
The proof of \eqref{localization} follows by considering
a cut--off function
\[
\eta_h(y)=\frac{1}{h} \min\{\max\{\rho+h-d(x,y),0\},1\}
\]
and defining $u_h=\eta_h \one_E$. Since
\[
\|Du_h\|(\Omega)\leq \frac{1}{h}\int_{
B_{\rho+h}(x)\backslash B_\rho(x)\cap \Omega} 
\one_E(y)\mathrm{d}\mu(y)
+\|D\one_E\|\left(B_{\rho+h}(x)\cap \Omega\right),
\]
passing to the limit as $h\to0$ (which is possible since we are assuming $m_E(x,\rho)$ to be differentiable with respect to $\rho>0$) and using the
lower semicontinuity of the total variation, 
we get \eqref{localization}.
If in particular $\|D\one_E\|\left(\Omega\cap \partial
B_\rho(x)\right)=0$, we then get
\[
\|D\one_{E\cap B_\rho(x)}\|\left(\partial B_\rho(x)\cap \Omega\right)
\leq m'_E(x,\rho).
\]

\bigskip

\section{Rough Trace}\label{roughtrace}

\medskip

In this section we extend the notion of rough trace of a $BV$ function to the metric measure space setting. The discussion will follow the monograph by V. Maz'ya \cite[Section 9.5]{Ma} and metric versions of the results contained therein will be given. In particular, we shall focus on the issue of the integrability of the rough trace with respect to the perimeter measure of the domain. We shall relate this issue with some geometric properties of the domain.

Below, $\Omega\subset\X$ shall always denote an open set. We always write $\|D\one_\Omega\|(\X)<\infty$ to intend $\one_\Omega\in BV(\X)$, and similarly, when $E\subset\Omega$, $\|D\one_E\|(\Omega)<\infty$ to intend $\one_E\in BV(\Omega)$, that is, to mean that the sets are of finite perimeter in $\X$ and $\Omega$, respectively.

\medskip

\begin{definition}\label{rough-def}{\normalfont{(Rough Trace)}}
Given $u\in BV(\Omega)$, we define its \textit{rough trace} at $x\in \partial^* \Omega$ as the quantity
\begin{equation}\label{rgh-tr}
u^{*}(x):=
\sup\Big\{ t\in\R: 
\|D\one_{E_t}\|(\X)<\infty,\:x\in\partial^{*}E_{t}
\Big\}.
\end{equation}
\end{definition}

Of course, when $u$ has a limit value at $x\in\partial^{*}\Omega$ from the interior of the domain, then 
\[
u^{*}(x)=\lim_{\Omega\ni y\to x} u(y).
\]


\begin{remark} We explicitly observe that, according to \eqref{rgh-tr}, it might occur that $u^*(x)=-\infty$, which obviously corresponds to the case when
\[
\Big\{ t\in\R: 
\|D\one_{E_t}\|(\X)<\infty,\:x\in\partial^{*}E_{t}
\Big\}=\emptyset. 
\]
\end{remark}


We start with the following result.

\smallskip

\begin{lemma}\label{Lemma31}
If $\|D\one_\Omega\|(\X)<\infty$ and $u\in BV(\Omega)$, then $u^{*}$ is ${\cal S}^{h}$-
measurable on $\partial^{*}\Omega$ and
\begin{equation}
{\cal S}^{h}\left(\left\{ x\in\partial^{*}\Omega:\:u^{*}(x)\ge t\right\} \right)={\cal S}^{h}\left(\partial^{*}\Omega\cap\partial^{*}E_{t}\right)\label{eq:2.1}
\end{equation}
for almost every $t\in\mathbb{R}$.
\end{lemma}

\begin{proof}
We fix a set $I\subset \R$ such that $|I|=0$ and $E_t$ has finite perimeter for any $t\in \R\backslash I$.
We define
\begin{equation*}
A_t\coloneqq\{ x\in\partial^*\Omega: u^*(x)\ge t\}\quad\mathrm{and}\quad B_t \coloneqq \partial^*E_t\cap\partial^*\Omega.
\end{equation*}
Observe that $B_t$ and $\partial^*\Omega$ are Borel sets, and that the definition of rough trace allows us to write
\begin{equation*}
 A_t = \bigcup_{s\in D;\:s>t} B_s\,,
\end{equation*}
for some countable, dense set $D\subset\R\backslash I$. Therefore, $A_t$ is a Borel set and then $u^*$ is a Borel function.

Now, instead of \eqref{eq:2.1}, we shall prove
that for every $t\in\R\backslash I$ - except at most countably many values - it holds
\[
{\cal S}^{h}\left(A_t \triangle B_t\right)=0,
\]
where $\triangle$ denotes the symmetric difference between two sets, $A\triangle B\coloneqq (A\backslash B)\cup(B\backslash A)$. If $x\in B_t$, the definition of $u^*$ implies that $u^*(x)\ge t$ and then the inclusion
$B_t\subset A_t$ holds. We then reduce ourselves to prove that ${\cal S}^{h}\left(F_{t}\right)=0$, where $F_{t}\coloneqq A_{t}\backslash B_{t}$. Since for $s<t$ we have that $E_t\subset E_s\subset \Omega$, we also get
\[
\frac{\mu(E_t\cap B_\rho(x))}{\mu(B_\rho(x))}
\leq
\frac{\mu(E_s\cap B_\rho(x))}{\mu(B_\rho(x))}
\leq
\frac{\mu(\Omega\cap B_\rho(x))}{\mu(B_\rho(x))},
\]
whence
\[
\Theta_*(E_t,x)\leq\Theta_*(E_s,x)\leq
\Theta^*(E_s,x)\leq \Theta^*(\Omega,x)
\]
and so the inclusion $B_t\subset B_s$ holds
true. From this we deduce that the sets $F_t$ are disjoint; indeed if $s<t$,
\[
F_t\cap F_s=(A_t\backslash B_t)\cap 
(A_s\backslash B_s)
=A_t\cap A_s \cap B_t^c\cap B_s^c
=A_t\cap B_s^c=A_t\backslash B_s=\emptyset
\]
since, if $x\in A_t$, then there exists $\tau\in [t,u^*(x)]$ such that $x\in \partial^*E_\tau\cap \partial^*\Omega=B_\tau\subset B_s$. 
The inclusion
$F_t\subset \partial^*\Omega$ then
implies that the set
\[
 \left\{t\in\mathbb{R}\backslash I:\:\mathcal{S}^h\left(F_t\right)>0\right\}
\]
is at most countable, and this concludes the proof.

\end{proof}

\smallskip

The result below is simply a combination of \cite[Section 9.5, Lemma 4 and Corollary 2]{Ma}, so we just state it, leaving the proof to the interested reader.

\begin{proposition} 
For any $u\in BV(\Omega)$ and for ${\cal S}^{h}$-almost
every $x\in\partial^{*}\Omega$, one has
\[
-u^{*}(x)=\left(-u\right)^{*}(x).
\]
Consequently, if we decompose $u=u^+-u^-$ in its positive
and negative part, then
\[
\left(u^{*}\right)^{+}=
\left(u^{+}\right)^{*},\qquad
\left(u^{*}\right)^{-}=\left(u^{-}\right)^{*}
\]
and then 
\[
u^{*}=\left(u^{+}\right)^{*}-\left(u^{-}
\right)^{*}.
\]
\end{proposition}


\begin{remark}\label{per-global}
Throughout the remainder of the paper, we shall always work in the hypothesis that $\Omega\subset\X$ is an open set with finite perimeter in $\X$; that is, $\|D\one_{\Omega}\|(\X)<\infty$.

Moreover, $E\subset\Omega$ will always be a Caccioppoli set in $\Omega$, which means $\|D\one_E\|(\Omega)<\infty$. We observe that, since in each of the next statements we shall assume $\partial\Omega\backslash\partial^*\Omega$ to be $\mathcal{S}^h$-negligible, this will imply that $\|D\one_E\|(\X)<\infty$ as well. Indeed, by Theorem \ref{locality}, $\|D\one_\Omega\|(\X)<\infty$ implies $\mathcal{S}^h(\partial^*\Omega)<\infty$, and therefore the condition $\mathcal{S}^h(\partial\Omega\backslash\partial^*\Omega)=0$ forces $\mathcal{S}^h(\partial\Omega)<\infty$; therefore, as $\|D\one_E\|(\Omega)<\infty$, an application of \cite[Proposition 6.3]{kkst1} yields $\|D\one_E\|(\X)<\infty$ as claimed. 
\end{remark}

\smallskip

In the next results, we will often make use of the following simple property of the rough trace:


\begin{remark}\label{rough-bdry}
 Let $\Omega\subset\X$ be such that $\|D\one_\Omega\|(\X)<\infty$. If $E\subset\Omega$ is a Caccioppoli set in $\Omega$, then $\one_E^*(x)=0$ for all $x\in\partial^*\Omega\backslash\partial^*E$ and $\one_E^*(x)=1$ for all $x\in\partial^*\Omega\cap\partial^*E$.
 
Indeed, when considering the characteristic function of $E$ one of course has
\[
E_{t}=\left\{ x\in\Omega:\:\mathds1_{E}\ge t\right\} =\begin{cases}
\emptyset, & t>1\\
\Omega, & t\le0\\
\Omega\cap E=E, & t\in (0,1].
\end{cases}
\]
This means, obviously,
\[
\partial^{*}E_{t}=\begin{cases}
\emptyset, & t>1\\
\partial^{*}\Omega, & t\le0\\
\partial^{*}\left(\Omega\cap E\right)=\partial^{*}E, & t\in (0,1].
\end{cases}
\]
So, when $t\in(0,1]$, the definition of rough trace \eqref{rgh-tr} forces $\one_E^*(x)=1$ for every $x\in\partial^*E$.

Let us then assume $t\le0$; again by \eqref{rgh-tr}, in order to have $x\in\partial^*E_t=\partial^*\Omega$, it must be $\one_E^*(x)=0$ for every $x$ therein. Thus, combining with the conclusion right above, we infer that $\one_E^*(x)=0$ for all $x\in\partial^*\Omega\backslash\partial^*E$ and $\one_E^*(x)=1$ for all $x\in\partial^*\Omega\cap\partial^*E$, proving the claim.
\end{remark}


With these preliminary facts at our disposal, we can start discussing the summability of the rough trace.


\begin{thm}
Let $\left\Vert D\mathds1_{\Omega}\right\Vert ({\X})<\infty$
and assume ${\cal S}^{h}\left(\partial\Omega\backslash\partial^{*}\Omega\right)=0$.
In order for any $u\in BV(\Omega)$ to satisfy
\begin{equation}\label{inf-l1-rough}
\inf_{c\in \R}
\int_{\partial\Omega}\left|u^{*}(x)-c\right|\mathrm{d}{\cal S}^{h}(x)\le k_1\left\Vert Du\right\Vert (\Omega)
\end{equation}
with $k_1>0$ independent of $u$, it is necessary and sufficient that the inequality
\begin{equation}\label{min-per}
\min\left\{\|D\one_E\|\left(\Omega^c\right), \|D\one_{\Omega\backslash E}\|\left(\Omega^c\right)\right\} 
\le 
k_2 \|D\one_E\|\left(\Omega\right)
\end{equation}
holds for any $E\subset\Omega$ with finite perimeter in $\Omega$.
\end{thm}

\begin{proof}
{\bf Necessity.} Let $E\subset\Omega$ be
such that ${\displaystyle \left\Vert D\mathds1_{E}\right\Vert (\Omega)}<\infty$,
and apply Remark \ref{per-global} 
to infer that $\|D\one_E\|(\X)<\infty$.
Then, since $\mathcal{S}^h (\partial\Omega\backslash\partial^*\Omega)=0$, by Remark \ref{rough-bdry} we get
\begin{equation*}
\underset{c\in\mathbb{R}}{\inf}\int_{\partial^{*}\Omega}\left|\mathds1_{E}^{*}(x)-c\right|\mathrm{d}{\cal S}^{h}(x) =\underset{c\in\mathbb{R}}{\min}\left\{ |1-c|{\cal S}^{h}\!\left(\partial^{*}E\cap\partial^{*}\Omega\right)+|c|{\cal S}^{h}\!\left(\partial^{*}\Omega\backslash\partial^{*}E\right)\right\}.
\end{equation*}
Now observe that the function
\begin{equation*}
 \Phi(c)\coloneqq |1-c|\mathcal{S}^h(\partial^*E\cap\partial^*\Omega)+|c|\mathcal{S}^h(\partial^*\Omega\backslash\partial^*E),
\end{equation*}
$c\in\mathbb{R}$, clearly attains its minima when $c=0$ and $c=1$ respectively, so we actually have
\begin{align*}
\underset{c\in\R}{\inf}\int_{\partial^*\Omega}\left|\one_E^*(x)-c\right|\text{d} \mathcal{S}^h(x) & =\mathrm{min} \left\{ \mathcal{S}^h(\partial^{*}\Omega\cap\partial^*E), \mathcal{S}^h(\partial^{*}\Omega\backslash\partial^{*}E)\right\} \\ 
&\ge \frac{1}{c_D}\,\mathrm{min}\left\{\|D\one_E\|(\Omega^c),\|D\one_{\Omega\backslash E}\|(\Omega^c)\right\}
\end{align*}
by Remark \ref{per-bound}.

Since by hypothesis
\[
{\displaystyle \underset{c\in\mathbb{R}}{\inf}\int_{\partial\Omega}\left|\mathds1_{E}^{*}(x)-c\right|\mathrm{d}{\cal S}^{h}(x)\le k_1\left\Vert D\one_E\right\Vert (\Omega),}
\]
we then obtain our claim.

\medskip

{\bf Sufficiency.}  Let $u\in BV(\Omega)$; then for every $t$, ${\cal S}^{h}\left(\partial\Omega\cap\partial^{*}E_{t}\right)$ is a non-increasing function of $t$. In fact, if ${\displaystyle x\in\partial^{*}\Omega\cap\partial^{*}E_{t}}$ and $\tau<t$, then ${\displaystyle \Omega\supset E_{\tau}\supset E_{t}}$ and the same holds as well for the essential boundaries; moreover,
\[
\Theta^*(E_t,x)\le\Theta^*(E_\tau,x)\le\Theta^*(\Omega,x).
\]
This means, by hypothesis and by the definition of essential boundary \eqref{dstar}, that ${\displaystyle x\in\partial^{*}\Omega\cap\partial^{*}E_{\tau}}$. In a similar manner we can show that ${\displaystyle {\cal S}^{h}\left(\partial\Omega\backslash\partial^{*}E_{t}\right)}$ is a non-decreasing function of $t$. By the Coarea Formula,  Remark \ref{per-bound} and \eqref{min-per}, there holds
\begin{align*}
k_2\left\Vert Du\right\Vert (\Omega)=&
k_2\int_{\mathbb{R}}\left\Vert D\mathds1_{E_{t}}\right\Vert (\Omega)\text{d}t\\
\ge& 
\alpha\int_{\mathbb{R}}\min\left\{ {\cal S}^{h}\left(\partial\Omega\cap\partial^{*}E_{t}\right),{\cal S}^{h}\left(\partial\Omega\backslash\partial^{*}E_{t}\right)\right\} \text{d}t.
\end{align*}
If we now set
\[\displaystyle t_{0}\coloneqq\sup\left\{ t:\:\left\Vert D\mathds1_{E_{t}}\right\Vert ({\X})<\infty,\:{\cal S}^{h}\left(\partial\Omega\cap\partial^{*}E_{t}\right)\geq{\cal S}^{h}\left(\partial\Omega\backslash\partial^{*}E_{t}\right)\right\},
\]
then we get, by recalling Lemma \ref{Lemma31},
\begin{align*}
k_2\left\Vert Du\right\Vert (\Omega) & \ge \alpha\left(\int_{t_{0}}^{+\infty}{\cal S}^{h}\left(\partial\Omega\cap\partial^{*}E_{t}\right)\text{d}t+\int_{-\infty}^{t_{0}}{\cal S}^{h}\left(\partial\Omega\backslash\partial^{*}E_{t}\right)\text{d}t\right) \\
 & =\alpha\bigg(\int_{t_{0}}^{+\infty}{\cal S}^{h}\left(\left\{ x\in\partial\Omega:u^{*}(x)\ge t\right\} \right)\text{d}t\\&\mskip+170mu+\int_{-\infty}^{t_{0}}{\cal S}^{h}\left(\left\{ x\in\partial\Omega:u^{*}(x)\le t\right\} \right)\text{d}t\bigg) \\
{\displaystyle } & =\alpha\left(\int_{\partial\Omega}\left[u^{*}(x)-t_{0}\right]^{+}\mathrm{d}{\cal S}^{h}(x)+\int_{\partial\Omega}\left[u^{*}(x)-t_{0}\right]^{-}\mathrm{d}{\cal S}^{h}(x)\right)\\
&=\alpha\int_{\partial\Omega}\left|u^{*}(x)-t_{0}\right|\mathrm{d}{\cal S}^{h}(x).
\end{align*}
In other words,
\[
\frac{k_2}{\alpha}\left\Vert Du\right\Vert (\Omega)=k'\|Du\|(\Omega)\ge\underset{c\in\mathbb{R}}{\inf}\int_{\partial\Omega}\left|u^{*}(x)-c\right|\mathrm{d}{\cal S}^{h}(x).
\]
\end{proof}

\smallskip

\begin{definition}\label{zeta-mazya} 
Let ${\displaystyle A\subset\overline{\Omega}}$. We shall denote by ${\displaystyle 0<\zeta_{A}^{(\alpha)}<\infty}$ the infimum of those $k>0$ such that ${\displaystyle \left[\left\Vert D\mathds1_{E}\right\Vert \left(\Omega^{c}\right)\right]^{\alpha}\le k\left\Vert D\mathds1_{E}\right\Vert (\Omega)}$ for all sets $E\subset\Omega$ which satisfy  the condition $\mu\left(E\cap A\right)+{\displaystyle \mathcal{S}^{h}\left(A\cap\partial^{*}E\right)}=0$.
\end{definition}

\smallskip

\begin{thm}\label{rough-zeta}
Let $\left\Vert D\mathds1_{\Omega}\right\Vert ({\X})<\infty$ and assume ${\cal S}^{h}\left(\partial\Omega\backslash\partial^{*}\Omega\right)=0$. Then, if $A\subset\overline{\Omega}$, for every $u\in 
BV(\Omega)$ such that ${\displaystyle u|
_{A\cap\Omega}=0}$ and ${\displaystyle u^{*}|_{A\cap\partial^{*}\Omega}=0}$, there is a constant $c>0$, depending on $\zeta^{(1)}_A$ and on $c_D$, such that
\[
{\displaystyle \int_{\partial\Omega}\left|u^{*}(x)\right|\mathrm{d}{\cal S}^{h}(x)\le c\left\Vert Du\right\Vert (\Omega)}.
\]
\end{thm}
\begin{proof}
By Cavalieri's Principle,

\smallskip

$\displaystyle{
\int_{\partial\Omega}\left|u^{*}(x)\right|
\mathrm{d}{\cal S}^{h}(x)}$
\begin{equation*}=\int_{0}^{+\infty}
\left[{\cal S}^{h}\left(\left\{ x\in\partial\Omega:u^{*}(x)\ge t\right\} 
\right)+{\cal S}^{h}\left(\left\{ x\in\partial\Omega:-u^{*}(x)\ge t
\right\} \right)\right]\text{d}t.
\end{equation*}

Notice that, by Lemma \ref{Lemma31} and Remark \ref{per-bound},
\begin{align*}
\int_{0}^{+\infty}{\cal S}^{h}\left(\{x\in\partial\Omega:u^*(x)\ge t\}\right)
\text{d}t & = \int_0^{+\infty}\mathcal{S}^h(\partial^*\Omega\cap\partial^*E_t)\text{d}t \\ &\le \frac{1}{\alpha}\int_{0}^{+\infty}\left\Vert D\mathds1_{E_{t}}
\right\Vert \left(\Omega^{c}\right)\text{d}t \\ & \le \frac{\zeta^{(1)}_A}{\alpha}\int_0^{+\infty}\|D\one_{E_t}\|(\Omega)\text{d}t,
\end{align*}
where we used the definition of $\zeta^{(1)}_A$ and the fact that, by our hypotheses, we get ${\displaystyle \mu\left(A\cap 
E_{t}\right)+\mathcal{S}^{h}\left(A\cap\partial^{*}E_{t}
\right)}=0$
for almost every $t>0$.

Similarly, we find
\begin{align*}
\int_{0}^{+\infty}{\cal S}^{h}\left(
\left\{ 
x\in\partial\Omega:-u^{*}(x)\ge t\right\} \right)\text{d}t
&\le
\frac{1}{\alpha}\int_{-\infty}^{0}
\left\Vert D\mathds1_{\Omega\backslash E_{t}}
\right\Vert \left(\Omega^{c}\right)\text{d}t\\
&\le
\frac{\zeta_{A}
^{(1)}}{\alpha}\int_{-\infty}^{0}\left\Vert D\mathds1_{E_{t}}
\right\Vert (\Omega)\text{d}t.
\end{align*}
Therefore, letting $c\coloneqq \dfrac{\zeta^{(1)}_A}{\alpha}$ gives the assertion.

\end{proof}

\smallskip

\begin{remark}
 In particular, if in Theorem \ref{rough-zeta} we substitute $u$ with $\mathds1_{E}\in BV(\Omega)$, $E\subset\Omega$, by Remark \ref{rough-bdry} we would simply have
\begin{align*}
 \|\one_E^*(x)\|_{L^1(\partial\Omega,\mathcal{S}^h)}& =\int_{\partial\Omega}|\one_E^*(x)|\text{d}\mathcal{S}^h(x) \\ & =\mathcal{S}^h(\partial^*\Omega\cap\partial^*E)\le \frac{1}{\alpha}\|D\one_E\|(\Omega^c) \\ & \le \frac{\zeta^{(1)}_A}{\alpha}\|D\one_E\|(\Omega)=c\|D\one_E\|(\Omega),
\end{align*}
where we explicitly used the assumption $\mathcal{S}^h(\partial\Omega\backslash\partial^*\Omega)=0$.
\end{remark}


The most important result of the present section is the following metric version of \cite[Theorem 9.5.4]{Ma}.


\begin{thm}\label{mazyathm}
Let $\Omega\subset\X$ be a bounded open set such that $\left\Vert D\mathds1_{\Omega}\right\Vert ({\X})<\infty$
and assume ${\cal S}^{h}\left(\partial\Omega\backslash\partial^{*}\Omega\right)=0$.
Then, every $u\in BV(\Omega)$ satisfies
\[
\left\Vert u^{*}\right\Vert _{L^{1}(\partial\Omega
,\mathcal{S}^h)}\le c\left\Vert u\right\Vert _{BV(\Omega)}
\]
with a constant $c>0$ independent of $u$, 
if and only if there exists $\delta>0$ such that for
every $E\subset\Omega$ with $\mathrm{diam}(E)\le\delta$ and with $\left\Vert D\mathds1_{E}
\right\Vert (\Omega)<\infty$
there holds 
\begin{equation}\label{per-compl}
\left\Vert D\mathds1_{E}\right\Vert \left(\Omega^{c}
\right)\le c'\left\Vert D\mathds1_{\text{E}}\right\Vert 
(\Omega)
\end{equation}
for some constant $c'>0$ independent of $E$.
\end{thm}
\begin{proof}
{\bf Necessity.} We start by recalling that by Remark \ref{rough-bdry}, one has $\one_E^*(x)=1$ on $\partial^*\Omega\cap\partial^*E$ and $\one_E^*(x)=0$ on $\partial^*\Omega\backslash\partial^*E$. Therefore,
\[
\left\Vert \mathds1_{E}^{*}\right\Vert _{L^{1}(\partial\Omega)}=\int_{\partial\Omega}\mathds1_{E}^{*}(x)\mathrm{d}{\cal S}^{h}(x)={\cal S}^{h}\left(\partial^{*}\Omega\cap\partial^{*}E\right),
\]
since by hypothesis ${\cal S}^{h}\left(\partial\Omega\backslash\partial^{*}\Omega\right)=0$.

\medskip

Now, let $\rho>0$ to be fixed in the sequel. We have the following

\bigskip

\begin{adjustwidth}{0.75cm}{0.75cm}
\textit{\underline{Claim.}} \textit{There exists $0<\delta<\rho$ such that,  for any $x_0\in \overline{\Omega}$,}
\[
\mu(B_\rho(x_0)\backslash B_\delta(x_0))
\ge \mu(B_\delta(x_0)).
\]
Assume  by contradiction that for any $\delta>0$
there exists $x_\delta\in  \overline{\Omega}$ such that
\[
\mu(B_\rho(x_\delta)\backslash B_\delta(x_\delta))
< \mu(B_\delta(x_\delta)).
\]
By taking $\delta=1/j$, $j\in\mathbb{N}$, we construct a  sequence
$(x_j)_{j\in\mathbb{N}}\subset\overline{\Omega}$ such that
\[
\mu(B_\rho(x_j)\backslash B_{1/j}(x_j))
< \mu(B_{1/j}(x_j)).
\]
By the compactness of $\overline{\Omega}$, up to subsequences we may assume $x_j\to x_0$. If we set  
\[
\varepsilon_j\coloneqq d(x_j,x_0),
\]
by the inclusions $B_{\rho-\varepsilon_j}(x_0)\backslash B_{1/j+\varepsilon_j}(x_0)\subset B_\rho(x_j)\backslash B_{1/j}(x_j)$ and $B_{1/j}(x_j)\subset B_{1/j+\varepsilon_j}(x_0)$ we would find that
\[
\mu(B_{\rho-\varepsilon_j}(x_0)\backslash
B_{1/j+\varepsilon_j}(x_0))< \mu(B_{1/j+\varepsilon_j}(x_0)).
\]
Passing to the limit as $j\to +\infty$, we obtain 
\[
\mu(B_\rho(x_0)\backslash\{x_0\})<\mu(\{x_0\}). 
\]
Since $\mu$ is  non-atomic, we  get a contradiction. The claim follows.
\end{adjustwidth}

\bigskip

Let now $E\subset \Omega\cap B_\delta(x_0)$ be a  set with finite perimeter; then
\begin{align}\label{min-muE}
\begin{split}
\mu(E) & = \mu(E\cap B_\delta(x_0))\\
&=\mu(E\cap B_\rho(x_0))
=\min\{ \mu(E\cap B_\rho(x_0)),\mu(B_\rho(x_0)\backslash E)\}.
\end{split}
\end{align}

If we consider the estimate 
\[
\| u^*\|_{L^1(\partial\Omega)}
\leq  c\| u\|_{BV(\Omega)}
\]
with $u=\mathds1_E$  we get, by the definition of the $BV$-norm,
\begin{equation}\label{rough-norm-est}
\|\mathds1_E^*\|_{L^1(\partial\Omega)}
\leq c \left(
\mu(E)+\|D\mathds1_E\|(\Omega)\right).
\end{equation}
Recall that under our hypotheses, $E$ has finite perimeter
also in $\X$ by Remark \ref{per-global}; therefore,
we also have the estimate
\begin{equation}\label{per-compl-est}
\|D\mathds1_E\|(\Omega^c)=
\|D\mathds1_E\|(\partial \Omega)
\leq c_D \mathcal{S}^h(\partial^*\Omega\cap \partial^*E)
=c_D \|\mathds1_E^*\|_{L^1(\partial\Omega)}.
\end{equation}
Applying the Poincar\'e inequality for $BV$ functions, we obtain
\begin{align}\label{poincare-BV}\begin{split}
\int_{B_{\rho}\left(x_{0}\right)}\left|\mathds1_{E}-\left(\mathds1_{E}\right)_{B_{\rho}\left(x_{0}\right)}\right|\text{d}\mu=&
\int_{B_{\rho}\left(x_{0}\right)}\left|\mathds1_{E}-\frac{\mu\left(E\cap B_{\rho}\left(x_{0}\right)\right)}{\mu\left(B_{\rho}\left(x_{0}\right)\right)}\right|\text{d}\mu \\
&\le c\rho\left\Vert D\mathds1_{E}\right\Vert \left(B_{\lambda\rho}\left(x_{0}\right)\right).\end{split}
\end{align}
Since $B_{\rho}\left(x_{0}\right)=E\cup\left(B_{\rho}\left(x_{0}\right)\backslash E\right)$,
computing the integral in \eqref{poincare-BV} gives

\medskip{}

${\displaystyle \mu\left(E\cap B_{\rho}\left(x_{0}\right)\right)\left(1-\frac{\mu\left(E\cap B_{\rho}\left(x_{0}\right)\right)}{\mu\left(B_{\rho}\left(x_{0}\right)\right)}\right)+\frac{\mu\left(E\cap B_{\rho}\left(x_{0}\right)\right)}{\mu\left(B_{\rho}\left(x_{0}\right)\right)}\mu\left(B_{\rho}\left(x_{0}\right)\backslash E\right)=}$
\begin{align*}
 & =\frac{\mu\left(E\cap B_{\rho}\left(x_{0}\right)\right)}{\mu\left(B_{\rho}\left(x_{0}\right)\right)}\mu\left(B_{\rho}\left(x_{0}\right)\backslash E\right)+\frac{\mu\left(E\cap B_{\rho}\left(x_{0}\right)\right)}{\mu\left(B_{\rho}\left(x_{0}\right)\right)}\mu\left(B_{\rho}\left(x_{0}\right)\backslash E\right)\\
 & =2{\displaystyle \frac{\mu\left(E\cap B_{\rho}\left(x_{0}\right)\right)}{\mu\left(B_{\rho}\left(x_{0}\right)\right)}\mu\left(B_{\rho}\left(x_{0}\right)\backslash E\right)}\\
 & =2\mu\left(E\cap B_{\rho}\left(x_{0}\right)\right)\left(1-\frac{\mu\left(B_{\rho}\left(x_{0}\right)\cap E\right)}{\mu\left(B_{\rho}\left(x_{0}\right)\right)}\right).
\end{align*}
As $\mu\left(B_{\delta}\left(x_{0}\right)\right)\le\mu\left(B_{\rho}\left(x_{0}\right)\backslash B_{\delta}\left(x_{0}\right)\right)$, by \eqref{min-muE} and
again by the Poincar\'e inequality we get
\begin{align*}
\mu\left(E\cap B_{\rho}\left(x_{0}\right)\right) & \le2\mu\left(E\cap B_{\rho}\left(x_{0}\right)\right)\left(1-\frac{\mu\left(B_{\rho}\left(x_{0}\right)\cap E\right)}{\mu\left(B_{\rho}\left(x_{0}\right)\right)}\right)\\
 & \le c\rho\left\Vert D\mathds1_{E}\right\Vert \left(B_{\lambda\rho}\left(x_{0}\right)\right)\\
 & =c\rho\left\Vert D\mathds1_{E}\right\Vert ({\X})\\
 & =c\rho\left(\left\Vert D\mathds1_{E}\right\Vert (\Omega)+\left\Vert D\mathds1_{E}\right\Vert \left(\partial\Omega\right)\right)\\
 & =c\rho\left(\left\Vert D\mathds1_{E}\right\Vert (\Omega)+\left\Vert D\mathds1_{E}\right\Vert \left(\Omega^{c}\right)\right),
\end{align*}
which, by the estimate $\left\Vert D\mathds1_{E}\right\Vert \left(\Omega^{c}\right)\le c_D\left(\left\Vert D\mathds1_{E}\right\Vert \left(\Omega\right)+\mu(E)\right)$
previously found by combining \eqref{rough-norm-est} and \eqref{per-compl-est}, entails
\begin{align*}
\left\Vert D\mathds1_{E}\right\Vert \left(\Omega^{c}\right) & \le c_D\left(\left\Vert D\mathds1_{E}\right\Vert (\Omega)+\mu\left(E\cap B_{\rho}\left(x_{0}\right)\right)\right)\\
 & \le c_D\left(\left\Vert D\mathds1_{E}\right\Vert (\Omega)+c\rho\left\Vert D\mathds1_{E}\right\Vert (\Omega)+c\rho\left\Vert D\mathds1_{E}\right\Vert \left(\Omega^{c}\right)\right)\\
 & =c_D\left(1+c\rho\right)\left\Vert D\mathds1_{E}\right\Vert \left(\Omega\right)+c_D\cdot c\rho\left\Vert D\mathds1_{E}\right\Vert \left(\Omega^{c}\right),
\end{align*}
whence
\[
\left\Vert D\mathds1_{E}\right\Vert \left(\Omega^{c}\right)\le c_D\frac{1+c\rho}{1-c_D\cdot c\rho}\left\Vert D\mathds1_{E}\right\Vert (\Omega)=c'\left\Vert D\mathds1_{E}\right\Vert (\Omega),
\]
where we of course require $\rho<\frac{1}{c_D\cdot c}$.

\medskip

{\bf Sufficiency.} Assume that \eqref{per-compl} holds
for any finite perimeter set $E\subset\Omega$
with diameter less then $\delta$.
This in particular implies that, by Remark \ref{per-bound},
\begin{align*}
{\cal{S}}^h(\partial^*E\cap \partial^*\Omega)
\leq&\frac{1}{\alpha}\|D\one_E\|\left(\partial^*\Omega\right)
=\frac{1}{\alpha}\|D\one_E\|(\partial\Omega) \\
=&\frac{1}{\alpha}\|D\one_E\|\left(\Omega^c\right)
\leq \frac{c'}{\alpha}\|D\one_E\|(\Omega).
\end{align*}
Let us fix then $u\in BV(\X)$ and assume 
$u\geq 0$; by Lemma~\ref{Lemma31} and Cavalieri's Principle, we obtain that
\begin{align*}\label{firstEst}\begin{split}
\| u^*\|_{L^1(\partial\Omega,{\cal S}^h)} &
=\int_0^\infty {\cal S}^h(\{x\in \partial^*\Omega:u^*(x)\geq t\}) \mathrm{d}t \\
& =\int_0^\infty
{\cal S}^h(\partial^* E_t\cap \partial^*\Omega) \mathrm{d}t.\end{split}
\end{align*}
Take $t\in [0,\infty)$ such that $E_t$
has finite perimeter in $\Omega$ and set
$E=E_t$. We fix $r>0$ such that $2r<\delta$ and consider a covering of $\X$ made of balls of the type 
$B_r(x_i)$, $i\in I\subset \N$,
such that $B_{2r}(x_i)$ have
overlapping bounded by $c_o>0$. 
We also select $r_i\in (r,2r)$ 
such that $m_E(x,\cdot)$ is differentiable at
$r_i$ and
\[
m_E'(x,r_i) \leq \frac{2 m_E(x,2r)}{r}.
\]
This is possible since 
$\rho\mapsto m_E(x,\rho)$ is monotone non decreasing and
\[
\int_r^{2r} m'_E(x,\rho) \d\rho 
\leq m_E(x,2r)-m_E(x,r)
\leq m_E(x,2r),
\]
so that
\[
\left\vert\left\{ t\in (r,2r): m'_E(x,t)>2\frac{m_E(x,2r)}{r}\right\}\right\vert
<\frac{r}{2}.
\]
We shall denote $B_i:=B_{r_i}(x_i)$.
Notice that for any measurable set $E$, we have $\partial^*E \cap B_i
\subset \partial^*(E\cap B_i)$. Indeed, for any
$x\in E\cap B_i$, there exists $\rho_0>0$
such that $B_\rho(x)\subset B_i$ for any 
$\rho<\rho_0$, hence
\[
\frac{\mu(E\cap B_i\cap B_\rho(x))}{
\mu(B_\rho(x))} = 
\frac{\mu(E\cap B_\rho(x))}{
\mu(B_\rho(x))}.
\]
Therefore, we have
\[
{\cal S}^h(\partial^* E\cap \partial^*\Omega) 
\leq \sum_{i\in I}
{\cal S}^h(\partial^* E\cap \partial^*\Omega\cap B_i)
\leq
\sum_{i\in I}
{\cal S}^h(\partial^* (E\cap B_i) \cap 
\partial^*\Omega).
\]
From this, using \eqref{localization}, Remark \ref{per-bound} and the fact that $r_i<\delta$, 
by assumption,
\begin{align*}
\mathcal{S}^h(\partial^*E\cap \partial^*\Omega) &
\leq
\sum_{i\in I} \mathcal{S}^h
(\partial^*(E\cap B_i)\cap \partial^*\Omega)
\leq \frac{1}{\alpha} 
\sum_{i\in I}\|D\one_{E\cap B_i}\|\left(\partial^{*}\Omega\right)\\
& =
\frac{1}{\alpha} 
\sum_{i\in I}\|D\one_{E\cap B_i}\|\left(\Omega^c\right)
\leq
\frac{c'}{\alpha} 
\sum_{i\in I}\|D\one_{E\cap B_i}\|(\Omega)\\ &
\leq
\frac{c'}{\alpha}\sum_{i\in I}
\left(
m'_E(x_i,r_i)+\|D\one_E\|(\Omega\cap\overline{B}_{r_i})
\right)\\ &
\leq
\frac{c'}{\alpha} 
\sum_{i\in I} 
\left(
\frac{m_E(x_i,2r)}{r}+\|D\one_E\|(\Omega\cap B_{2r}(x_i))
\right) \\ &
\leq \frac{c'c_o}{\alpha} \left(
\mu(E\cap \Omega)+\|D\one_E\|(\Omega)
\right).
\end{align*}
So, recalling that $E=E_t$, we have just obtained the
estimate
\[
\mathcal{S}^h(\partial^* E_t\cap \partial^*\Omega)
\leq 
\frac{c'c_o}{\alpha}
\left(
\mu(E_t\cap \Omega)+\|D\one_{E_t}\|(\Omega)
\right).
\]
Integrating this inequality and using Coarea
formula, we conclude that
\[
\|u^*\|_{L^1(\partial\Omega, \mathcal{S}^h)}
\leq
\frac{c'c_o}{\alpha}
\Big(
\|u\|_{L^1(\Omega)}+\|Du\|(\Omega)
\Big).
\]
The general case $u\in BV(\Omega)$ can be done by splitting $u=u^+-u^-$ into its 
positive and negative part.

\end{proof}


It is worth observing that the condition $\delta<\rho<\frac{1}{c_D\cdot c}$ found in the proof of Theorem \ref{mazyathm} tells us that the nature of this result is very local, as it holds at sufficiently small scales only.

\smallskip

We end this discussion by considering the issue of the extendability of a $BV$ function by a constant in terms of its rough trace. For this purpose, we first re-adapt the main arguments of \cite{Ma} and then discuss an alternative result for the zero-extension of a function $u\in BV(\Omega)$ to the whole of $\X$.

\smallskip

\begin{definition}
Let $\Omega\subset\X$ be an open set and let $u\in BV(\Omega)$. We define its $\beta$\textit{-extension} to $\X$, $\beta\in\R$, by setting
\[
u_{\beta}(x)\coloneqq\begin{cases}
u(x), & x\in\Omega\\
\\
\beta, & x\in\Omega^{c}.
\end{cases}
\]
\end{definition}

We then have the following:


\begin{lemma}\label{rough-ext} Assume $\Omega\subset\X$ is an open set such that $\|D\one_{\Omega}\|(\X)<\infty$ and $\mathcal{S}^h\left(\partial\Omega\backslash\partial^{*}\Omega\right)=0$. Let $\beta\in\R$ and $u\in BV(\Omega)$. Then, one has
\begin{equation*}
 \|Du_\beta\|(\X)\le\|Du\|(\Omega)+c_D\|(u-\beta)^*\|_{L^1(\partial\Omega,{\cal S}^h)}.
\end{equation*}
\end{lemma}
\begin{proof}
By the Coarea Formula, one obviously has
\[
 \|Du\|(\Omega)=\int_{\R} \|D\one_{E_t}\|(\Omega)\text{d}t.
\]
Since any two functions differing by an additive constant have the same total variation, the following holds:
\begin{equation}\label{du-pm-const}
 \|Du\|(\Omega)=\|D(u-\beta)\|(\Omega)=\int_\R\|D\one_{\{u-\beta\ge t\}}\|(\Omega)\d t.
\end{equation}
Therefore, when computing the total variation of $u_\beta$ on $\X$ one is obviously entitled to write
\begin{align}\label{Dubeta}\begin{split}
 \|Du_\beta\|(\X) & =\int_{\R}\|D\one_{\{u_\beta -\beta\ge t\}}\|(\X)\d t\\
 & =\int_\R \|D\one_{\{u_\beta -\beta\ge t\}}\|(\Omega)\d t+\int_\R\|D\one_{\{u_\beta -\beta\ge t\}}\|(\Omega^c)\d t\\
 & =\int_\R\|D\one_{\{u-\beta\ge t\}}\|(\Omega)\d t+\int_\R\|D\one_{\{u_{\beta} -\beta\ge t\}}\|(\partial\Omega)\d t\\&=\|Du\|(\Omega)+\int_\R\|D\one_{\{u-\beta\ge t\}}\|(\partial\Omega)\d t,\end{split}
\end{align}
where we made use of \eqref{du-pm-const} together with the facts that $u_\beta-\beta=u-\beta$ on $\Omega$ and that the total variation of $u_\beta$ at the boundary $\partial\Omega$ takes into account the ``jump'' of $u$ therein (namely, how $u_\beta=u$ and $u_\beta=\beta$ ``join'' at $\partial\Omega$).

Let us then estimate the second term at the rightmost side of \eqref{Dubeta}; by Remark \ref{per-bound} and Lemma \ref{Lemma31} we obtain
\begin{align*}
 \int_\R \|D\one_{\{u-\beta\ge t\}}\|(\partial\Omega)\d t \\
 & = \int_{-\infty}^0\|D\one_{\{u-\beta\ge t\}}\|(\partial\Omega)\d t+\int_0^{+\infty}\|D\one_{\{u-\beta\}}\|(\partial\Omega)\d t \\
 & \le c_D\bigg(\int_{-\infty}^0 \mathcal{S}^h(\{x\in\partial\Omega;\:(u-\beta)^{*}(x)\le t\})\d t \\
 & \mskip+85mu+\int_0^{+\infty}\mathcal{S}^h(\{x\in\partial\Omega;\:(u-\beta)^{*}(x)\ge t\})\d t\bigg) \\
 & =c_D\int_{\partial\Omega}|(u-\beta)^*|\d\mathcal{S}^h=c_D\|(u-\beta)^*\|_{L^1(\partial\Omega,\mathcal{S}^h)}.
\end{align*}

\end{proof}

\smallskip

\begin{remark}
 It is clear that Lemma \ref{rough-ext} gives an upper bound for $\|Du_\beta\|(\X)$, but without any further assumptions \textit{it does not allow us to conclude that} $u_\beta \in BV(\X)$. However, if we reformulate the statement in the hypotheses of Theorem \ref{mazyathm}, then it turns out that the zero-extension of $u\in BV(\Omega)$ to the whole of $\X$, $u_0$, has $BV$ norm $\|u_0\|_{BV(\X)}$ bounded by the $BV$ norm of $u$ in $\Omega$. In other words, $u_0\in BV(\X)$.
\end{remark}


Actually, by assuming the function $u$ to be also essentially bounded, it is possible to get $u_0\in BV(\X)$ under weaker hypotheses:


\begin{proposition}\label{bv-ext} Assume $\Omega\subset\X$ is an open set such that $\|D\one_\Omega\|(\X)<\infty$ and 
$\mathcal{S}^{h}\left(\partial\Omega\backslash\partial^{*}\Omega\right)$ $=0$;
let $E\subset\Omega$ be such that $\|D\one_E\|(\Omega)<\infty$. Then, one has $\|D\one_E\|(\X)<\infty$.

Under the same assumptions, for any $u\in BV\cap L^\infty(\Omega)$
one has $u_0\in BV(\X)$.
\end{proposition}

\begin{proof}
The first part of the statement can be actually seen as a particular case of \cite[Proposition 6.3]{kkst1}, so we refer to our previous Remark \ref{per-global} for more comments. 

Let us take $u\in BV\cap L^\infty(\Omega)$ and let us start by first assuming
$u\geq 0$.
Then, since for any $t>0$
\[
\bar{E}_t\coloneqq\{ u_0\geq  t\} =\{x\in \Omega: u(x)\geq t\}=E_t, 
\]
we obtain that
\begin{align*}
\left\Vert Du_0\right\Vert (\X)=&
\int_0^\infty \|D\mathds1_{\{u_0\geq t\}}\| (\X)\text{d}t=
\int_{0}^{+\infty}\left\Vert D{\mathds1}_{{E}_{t}}\right\Vert (\X)\text{d}t\\
\le&
\int_{0}^{+\infty}\left[\left\Vert D\mathds1_{E_{t}}\right\Vert (\Omega)+c_D\,\mathcal{S}^{h}\left(\partial^{*}\Omega\cap\partial^{*}\bar{E}_{t}\right)\right]\text{d}t.
\end{align*}
Since $u\in L^\infty(\Omega)$, we can consider a Borel representative
of $u$ such that for any $t>\|u\|_\infty$, $E_t=\emptyset$; then
we obtain the estimate
\[
\|Du_0\|(\X)\leq
\|Du\|(\Omega)+c_D  \mathcal{S}^h(\partial^* \Omega)  \|u\|_\infty,
\]
whence $u_0\in BV(\X)$. 

The general case $u\in BV\cap L^\infty(\Omega)$ follows by considering the decomposition $u=u^+-u^-$ into its
positive and negative part.

\end{proof}

\smallskip

\begin{remark}\label{aplim-ext-bdd} 
 We recall that in \cite[Lemma 3.2]{kkst2} it was proved that for any function $u\in BV(\X)$ its approximate limits satisfy
 \[
  -\infty<u^\wedge(x)\le u^\vee(x)<\infty
 \]
for $\mathcal{S}^h$-almost every $x\in\X$. 

Consequently, if we assume the hypotheses of Proposition \ref{bv-ext} to be satisfied, or if we re-state Lemma \ref{rough-ext} including the hypotheses of Theorem \ref{mazyathm}, we can conclude that
\[
-\infty<{u_0}^\wedge(x)\le{u_0}^\vee(x)<\infty
 \]
for $\mathcal{S}^h$-almost every $x\in\X$.
\end{remark}

\medskip

\subsection{An Integration by Parts Formula for 
\texorpdfstring{$BV$}{BV} functions}

\medskip

Summarizing the previous results, we can state that
\begin{equation*}
 \|D\one_\Omega\|(\X)<\infty\qquad\mathrm{and}\qquad\mathcal{S}^h\left(\partial\Omega\backslash\partial^*\Omega\right)=0
\end{equation*}
are the underlying conditions for the domain $\Omega$ which, thanks to Theorem \ref{mazyathm}, ensure that the rough trace $u^{*}(x)$
of any function $u\in BV(\Omega)$ is in $L^{1}\left(\partial\Omega,\mathcal{S}^h\right)$. 

This conclusion motivates us, as already done in \cite{bu}, to proceed towards an integration by parts formula for functions of bounded variation by means of a suitable class of \textit{vector fields}.

To this aim, we shall refer to the characterization of $BV$ functions given in \cite{di} by means of \textit{Lipschitz derivations}: na\"ively, linear operators acting on Lipschitz functions and satisfying a Leibniz rule. This class of objects was previously introduced, in more generality, by N. Weaver in his seminal paper \cite{we}. In \cite{di}, Lipschitz derivations served as the ideal tool to define a $BV$ space via an integration by parts formula and to recover a familiar representation formula for the total variation of a $BV$ function on any domain $\Omega\subset\X$. Moreover, as shown in \cite[Theorem 7.3.7]{di}, this version of the $BV$ space turns out to be equivalent to the ``relaxed'' one of \cite{Mir} that we are using in the present work. We also observe that this equivalence is independent of structural assumptions on the ambient space, as \cite{di} operated in the context of a complete, separable metric measure space $(\X,d,\mu)$ endowed with a non-negative Radon measure $\mu$ giving finite mass to bounded sets. Here, however, for our purposes we continue to assume that $(\X,d,\mu)$ is a doubling metric measure space supporting a weak $(1,1)$--Poincar\'e inequality.

We now recall, after \cite{di}, the basic notions and properties regarding Lipschitz derivations. We shall also present the construction of the respective $BV$ space and discuss the essential properties of this characterization, including the equivalence with the definition via relaxation. Observe that, while in \cite{di} Lipschitz derivations are taken to act on Lipschitz functions with bounded support, $\lip_{\rm bs}(\X)$, here we shall define them on the space $\lip_c(\X)$ of compactly supported Lipschitz maps, since in our case the ambient space is proper and then the two classes of functions coincide. The symbol $L^0(\X)$ will be used to denote the space of $\mu$-measurable functions (with no summability requirements), while $\mathbf{M}(\X)$ will indicate the space of finite, signed Radon measures on $\X$.

\bigskip

\begin{definition}\label{lip-der}By a Lipschitz derivation we shall intend any linear map $\dd:\lip_c(\X)\to L^0(\X)$ satisfying the following:
\begin{enumerate}
 \item {\normalfont{\underline{Leibniz rule.}}} For any $f,g\in\lip_c(\X)$, there holds $\dd(fg)=\dd(f)g+f\dd(g)$.
 \item {\normalfont{\underline{Weak locality.}}} There exists some function $g\in L^0(\X)$ such that, for all $f\in\lip_c(\X)$,
 \begin{equation}\label{weak-loc}
  |\dd(f)|\le g\cdot\lip_a(f)\quad\mu\text{-almost\:everywhere}.
 \end{equation}
The smallest function $g$ satisfying \eqref{weak-loc} will be denoted by $|\dd|$.
\end{enumerate}
\end{definition}

In the weak locality condition, $\lip_a(f)$ denotes the \textit{asymptotic Lipschitz constant}, defined as
\begin{equation*}
 \lip_\text{a}(f)(x)\coloneqq\underset{\rho\to 0}{\lim}\,\lip(f,B_\rho(x)),
\end{equation*}
for any ball $B_\rho(x)\subset\X$, where, for any set $E\subset\X$, we define
\begin{equation*}
 \lip(f,E)\coloneqq\underset{x,y\in E;\:x\neq y}{\sup}\frac{|f(x)-f(y)|}{d(x,y)}.
\end{equation*}
 The set of all Lipschitz derivations in $\X$ will be denoted by $\der(\X)$. Moreover, we shall write $\dd\in L^p(\X)$ to intend $|\dd|\in L^p(\X)$.
 
 
 Together with the above notion of derivations, one can define the \textit{divergence} by imposing an integration by parts formula.
 
 \begin{definition}
  \label{def-div}Let $\dd\in L^1_\loc(\X)$. We define the divergence of $\dd$, written $\div(\dd)$, as the linear operator such that
  \begin{equation*}
   \lip_c(\X)\ni f\mapsto-\int_\X\dd(f)\d\mu.
  \end{equation*}
 \end{definition}

 We shall write $\div(\dd)\in L^p(\X)$, $p\in[1,\infty]$, when this operator admits an integral representation via an $L^p$ map: $\div(\dd)=h\in L^p(\X)$ if
 \begin{equation*}
  -\int_\X\dd(f)\d\mu=\int_\X fh\d\mu=\int_\X f\div(\dd)\d\mu.
 \end{equation*}
We observe that $\div(\dd)\in L^p(\X)$ implies its uniqueness. Together with $\der(\X)$, for $p,q\in[1,\infty]$ let us also consider the spaces
\begin{equation*}
 \der^p(\X)\coloneqq\Big\{\dd\in\der(\X):\;\dd\in L^p(\X)\Big\},
\end{equation*}
and
\begin{equation*}
 \der^{p,q}(\X)\coloneqq\Big\{\dd\in\der(\X):\;\dd\in L^p(\X),\:\div(\dd)\in L^q(\X)\Big\}.
\end{equation*}
In particular, we shall concentrate on $\der_b(\X)\coloneqq\der^{\infty,\infty}(\X)$, namely the subspace of \textit{bounded Lipschitz derivations} which will be used in the definition of $BV_\der$ below.


\begin{remark}\label{homog-der}
We observe that by \cite[Lemma 7.1.2]{di}, for any $\dd\in\der(\X)$ and any $u\in L^0(\X)$ there holds $|u\dd|=|u|\cdot|\dd|$. In particular, if $u\in\lip_b(\X)$ and $\dd\in\der^{p,q}(\X)$, $p,q\in[1,\infty]$, then $u\dd\in\der(\X)$ is such that
\begin{equation*}
 \div(u\dd)=u\div(\dd)+\dd(u)\quad\text{and}\quad u\dd\in\der^{p,r}(\X),
\end{equation*}
with $r=\max\{p,q\}$.
\end{remark}


We can now define the space of $BV$ functions by means of Lipschitz derivations.

\begin{definition}
 \label{def-bv-der}Let $u\in L^1(\X)$. We say that $u\in BV_\der(\X)$ if there exists a linear and continuous operator $L_u:\der_b(\X)\to\M(\X)$ such that
 \begin{equation}
  \label{eq-BV} \int_\X\d L_u(\dd)=-\int_\X u\div(\dd)\d\mu\quad\forall\,\dd\in\der_b(\X),
 \end{equation}
such that $L_u(h\dd)=hL_u(\dd)$ for every $h\in\lip_b(\X)$ and every $\dd\in\der_b(\X)$.
\end{definition}


\begin{remark}
 We observe that Definition \ref{def-bv-der} is well posed, in the sense that it does not depend on the particular map $L_u$ realizing \eqref{eq-BV}. To see this, for $u\in BV_\der(\X)$ and $\dd\in\der_b(\X)$, let $L_u$ and $\tilde{L}_u$ be any two maps as in the definition of $BV_\der(\X)$. Then, for any $h\in\lip_b(\X)$, apply Remark \ref{homog-der} to find that $h\dd\in\der_b(\X)$, so by \eqref{eq-BV} and by the Lipschitz-linearity we get
 \begin{equation*}
  \int_\X  h\d L_u(\dd)=\int_\X \d L_u(h\dd)=-\int_\X u\div(h\dd)\d\mu,
 \end{equation*}
and the same holds with $\tilde{L}_u$ in place of $L_u$. In particular,
\begin{equation*}
 \int_\X h\d L_u(\dd) = \int_\X h\d\tilde{L}_u(\dd),
\end{equation*}
so by the arbitrariness of $h\in\lip_b(\X)$ we get $L_u(\dd)=\tilde{L}_u(\dd)$.\\ With a slight abuse of notation, for $u\in BV_\der(\X)$ this common value will be denoted as $\d u(\dd)$.
\end{remark}


\begin{proposition}\label{weak-var}
 {\normalfont{\cite[Theorem 7.3.3]{di}}} Let $u\in BV_\der(\X)$. Then, there exists a finite, non-negative Radon measure $\nu$ on $\X$ such that, for every Borel set $B\subset\X$, one has
 \begin{equation}\label{weak-tot-var}
  \int_B \d u(\dd)\le\int_B|\dd|^* \d\nu,
 \end{equation}
where $|\dd|^*$ denotes the upper semicontinuous envelope of $|\dd|$. The smallest measure satisfying \eqref{weak-tot-var} will be denoted by $\|\d u(\dd)\|$, the weak total variation of $\d u(\dd)$. Moreover,
\begin{equation*}
 \|\d u(\dd)\|(\X)=\sup\Big\{|\d u(\dd)(\X)|:\;\dd\in\der_b(\X),\:|\dd|\le1\Big\}.
\end{equation*}
\end{proposition}


As one may expect, this definition of $BV$ via integration by parts allows for a familiar representation formula for the weak total variation:


\begin{thm}\label{rep-form}
 {\normalfont{\cite[Theorem 7.3.4]{di}}} Let $u\in BV_\der(\X)$. Then, for every open set $\Omega\subset\X$,
 \begin{equation*}
  \|\d u(\dd)\|(\Omega)=\sup\left\{\int_\Omega f\div(\dd)\d\mu:\;\dd\in\der_b(\X),\:|\dd|\le1,\:\mathrm{supp}(\dd)\Subset\Omega\right\}.
 \end{equation*}
 
\end{thm}

\smallskip

The set of bounded Lipschitz derivations $\dd\in\der_b(\X)$ such that $\mathrm{supp}(\dd)\Subset\Omega$ will be denoted by $\der_b(\Omega)$.


It turns out that the definition of $BV_\der$ via Lipschitz derivations produces the same $BV$ space introduced by \cite{Mir}.


\begin{thm}\label{bv-equiv}
 {\normalfont{\cite[Theorem 7.3.7]{di}}} One has the equivalence
 \begin{equation*}
  BV(\X)=BV_\der(\X).
 \end{equation*}
In particular, the respective notions of total variation coincide, so for every $u\in BV(\X)$, there holds
\begin{equation}\label{equal-var}
 \|Du\|(\X)=\|\d u(\dd)\|(\X).
\end{equation}
for all $\dd\in\der_b(\X)$ with $|\dd|\le1$.
\end{thm}

Of course, Theorem \ref{bv-equiv} and in particular \eqref{equal-var} continue to hold true - with the obvious readaptations - if we replace $\X$ with any open set $\Omega\subset\X$.

\smallskip

\begin{remark}\label{prop-du}Let $u\in BV(\Omega)$, $\Omega\subset\X$ open set. We want to discuss some easy properties of the measure $\d u(\dd)$, $\dd\in\der_b(\Omega)$.
\begin{enumerate}
 \item \label{prop-du-1} We first observe that Proposition \ref{weak-var} combined with Theorems \ref{rep-form}-\ref{bv-equiv} immediately yields the absolute continuity of $\d u(\dd)$ with respect to $\|Du\|$.
Therefore, by the Radon-Nikod\'ym Theorem there exists a density $F^{\dd}\in L^1(\|Du\|)$ such that $\d u(\dd)=F^{\dd}\|Du\|$.

Let now $E\subset\Omega$ be such that $\|D\one_E\|(\Omega)<\infty$. Then, an application of Theorem \ref{locality} gives
\begin{align*}
 \d\one_E(\dd)(\Omega^c) & = \int_{\Omega^c} \d\one_E(\dd) = \int_{\Omega^c} F^{\dd}_E\,\d\|D\one_E\| \\ & = \int_{\Omega^c\,\cap\,\partial^*E} F_E^{\dd}\,\theta_E\,\d{\cal S}^h = \int_{\partial\Omega\,\cap\,\partial^*E} F_E^{\dd}\,\theta_E\,\d{\cal S}^h\\
 & =\int_{\partial\Omega}F_E^{\dd}\,\d\|D\one_E\| =\int_{\partial\Omega}\d\one_E(\dd)=\d\one_E(\dd)(\partial\Omega),
\end{align*}
where we wrote $F_E^{\dd}$ to keep into account the fact that, a priori, this density might depend also on the set where we localize.
\item \label{prop-du-2}Let $B\subset\Omega$ be a Borel set. By combining Cavalieri's Principle with Fubini's Theorem, we find
\begin{align*}\begin{split}
 \d u(\dd)(B) = \int_B \d u(\dd) & = - \int_B u\div(\dd)\d\mu \\
 & = - \int_B\left(\int_{-\infty}^{+\infty} \one_{E_t}\d t\right)\div(\dd)\d\mu \\
 & = -\int_{-\infty}^{+\infty}\int_B\one_{E_t}\div(\dd)\d\mu\d t \\
 & = \int_{-\infty}^{+\infty}\d\one_{E_t}(\dd)(B)\d t.\end{split}
\end{align*}
In particular, when $B=\Omega$, the above clearly becomes
\begin{equation*}
 \d u(\dd)(\Omega) = \int_{-\infty}^{+\infty} \d \one_{E_t}(\dd)(\Omega)\d t.
\end{equation*}
\end{enumerate}
\end{remark}

\smallskip
 
\begin{thm}\label{gauss-green} Let $\Omega\subset\X$ be a bounded open set such that $\|D\one_\Omega\|(\X)<\infty$ and ${\cal S}^h(\partial\Omega\backslash\partial^*\Omega)=0$. Then, for every $u\in BV(\Omega)$ and every $\dd\in\der_b(\X)$, one has
\begin{equation}\label{int-part}
 \int_\Omega \d u(\dd)+\int_\Omega u\div(\dd)\d\mu=-\int_{\partial\Omega}\Theta^{\dd}_{E_t}(u^*(x))\d{\cal S}^h(x)
\end{equation}
for some function $\Theta^{\dd}_{E_t}\in L^1(\partial\Omega,{\cal S}^h)$, where, as usual, $E_t\coloneqq\{x\in\Omega;\:u(x)\ge t\}$.
\end{thm}
\begin{proof}
 Let us first start with $u=\one_E$, where $E\subset\Omega$ is a Caccioppoli set in $\Omega$. Then, writing $\Omega=\X\backslash\Omega^c$, one gets, by the property \ref{prop-du-1} in Remark \ref{prop-du} and by noticing that the hypotheses grant that $\one_E\in BV(\X)$ by Remark \ref{per-global},
 \begin{align}\label{d-oneE-Om}\begin{split}
  \int_\Omega \d\one_E(\dd) & = \int_\X \d\one_E(\dd)-\int_{\Omega^c} \d\one_E(\dd) \\
  & = - \int_\X \one_E\div(\dd)\d\mu - \int_{\partial\Omega} \d\one_E(\dd).\end{split}
 \end{align}
Observe that the second integral at the rightmost side in \eqref{d-oneE-Om} is actually on $\partial\Omega\cap\partial^*E$ since, by \ref{prop-du-1} in Remark \ref{prop-du}, $\d\one_E(\dd)\ll\|D\one_E\|$ and the latter is concentrated on $\partial^*E$ by Theorem \ref{locality}.

\smallskip

Now, assume $0\le u\in BV(\Omega)$ for simplicity; the proof for a general function $u\in BV(\Omega)$ will follow by splitting $u$ into its positive and negative parts. \\Since $u\in BV(\Omega)$, we can use both the properties \ref{prop-du-1}--\ref{prop-du-2} in Remark \ref{prop-du} and Theorem \ref{locality} to rewrite \eqref{d-oneE-Om} as
\begin{align*}\begin{split}
 \int_\Omega \d u(\dd) \\ & = \int_0^{+\infty}\d\one_{E_t}(\dd)(\Omega)\d t = \int_0^{+\infty} \d t\int_\Omega \d\one_{E_t}(\dd) \\
 & = -\int_0^{+\infty} \d t\left(\int_\X \one_{E_t}\div(\dd)\d\mu + \int_{\partial\Omega} \d\one_{E_t}(\dd)\right) \\
 & = - \int_0^{+\infty} \d t\left(\int_\Omega \one_{E_t}\div(\dd)\d\mu + \int_{\partial^*\Omega\cap\partial^*E_t} F^{\dd}_{E_t} \theta_{E_t} \d{\cal S}^h\right)\\
 & = -\int_\Omega u\div(\dd)\d\mu-\int_0^{+\infty}\d t\left(\int_{\partial^*\Omega\cap\partial^*E_t} F^{\dd}_{E_t} \theta_{E_t} \d{\cal S}^h\right)\\
 & = -\int_{\Omega} u\div(\dd)\d\mu-\int_0^{+\infty}\d t\left(\int_{\{x\in\partial^*\Omega;\:u^*(x)\ge t\}} F^{\dd}_{E_t} \theta_{E_t}\d{\cal S}^h\right) \\
 & = -\int_\Omega u\div(\dd)\d\mu - \int_{\partial\Omega} \Theta^{\dd}_{E_t}(u^*(x))\d{\cal S}^h,\end{split}
\end{align*}
where we exploited the fact that ${\cal S}^h(\partial^*\Omega\cap\partial^*E_t)={\cal S}^h(\{x\in\partial^*\Omega;\:u^*(x)\ge t\})$ by Lemma \ref{Lemma31} and we defined
\begin{equation*}
 \Theta^{\dd}_{E_t}(u^*(x))\coloneqq\int_0^{u^*(x)} F^{\dd}_{E_t}\theta_{E_t} \d t.
\end{equation*}
Therefore, \eqref{int-part} follows.\\
\end{proof}

\smallskip

\begin{remark}\label{rmk-gauss} Let us add some comments on Theorem \ref{gauss-green}.
 \begin{enumerate}
\item We observe that the property \ref{prop-du-1} in Remark \ref{prop-du} obviously holds for $\d\one_\Omega(\dd)$ as well, so that, besides the localization on $\partial^*E$
\begin{equation*}
 \d\one_E(\dd) = F^{\dd}_E\,\theta_E\,{\cal S}^h\mrestr\partial^*E
\end{equation*}
for $\d\one_E(\dd)$, always by virtue of Theorem \ref{locality} there holds an analogous localization for $\d\one_\Omega(\dd)$: that is,
\begin{equation*}
 \d\one_\Omega(\dd) = F^{\dd}_\Omega\, \theta_\Omega\, {\cal S}^h\mrestr\partial^*\Omega.
\end{equation*}

In the same spirit of Definition \ref{local-space}, we shall say that $({\X},d,\mu)$
is \textit{\textcolor{black}{strongly local }}\textcolor{black}{if,
together with the condition $\theta_{E}=\theta_{\Omega}$ ${\cal S}^{h}$}--almost
everywhere on $\partial^{*}\Omega\cap\partial^{*}E$, one also has
\begin{equation*} F_E^{\dd} = F_\Omega^{\dd}
 \qquad\mathcal{S}^h\mathrm{-almost\;everywhere\;on}\;\partial^*\Omega\cap\partial^*E.
\end{equation*}

\item We can apply Theorem \ref{gauss-green} to the special case where
$\Omega$ is a regular domain in the sense of \cite{bu}: that is, an open set of finite perimeter coinciding with the upper inner Minkowski content of its boundary,
\begin{equation*}
 \|D\one_{\Omega}\|(\X)=\mathfrak{M}^{*}_{\mathrm{in}}(\partial\Omega)\coloneqq\underset{t\to 0}{\lim\sup}\frac{\mu\left(\Omega\backslash\Omega_t\right)}{t}.
 \end{equation*}
Here, by $\Omega_t$  we intend the super-level sets of the distance function: \begin{equation*}\Omega_t\coloneqq\left\{x\in\Omega:\:\mathrm{dist}\left(x,\Omega^c\right)\ge t\right\},\end{equation*}$t>0$. Thus said, if we also require that ${\cal S}^{h}\left(\partial\Omega\backslash\partial^{*}\Omega\right)=0$, we can show that for every $u\in BV(\Omega)$ there exists a \textit{trace operator}
\begin{equation*}\mathrm{T}:BV(\Omega)\rightarrow L^{1}\left(\partial\Omega,\left\Vert D\mathds1_{\Omega}\right\Vert \right)\end{equation*}
such that, for every $\dd\in\der_b(\X)$, one has
\begin{align*}
\int_{\Omega}\d u(\dd)+\int_{\Omega}u\,\div(\dd)\,\mathrm{d}\mu\\ & = - \int_{\partial\Omega}u^{*}(x)\left(\dd\cdot\nu\right)_{\partial\Omega}^{-}\mathrm{d}\left\Vert D\mathds1_{\Omega}\right\Vert\\&\coloneqq\left\langle \mathrm{T}u,\left(\dd\cdot\nu\right)_{\partial\Omega}^{-}\right\rangle.
\end{align*}
Here, the map $(\dd\cdot\nu^-_{\partial\Omega})\in L^1(\partial\Omega,\|D\one_\Omega\|)$ is the \textit{inner normal trace} of $\dd$ on $\partial\Omega$. 
Indeed, in this case we can use the \textit{defining sequence} $\left(\varphi_{\varepsilon}\right)_{\varepsilon>0}\subset\text{Lip}_{c}(\Omega)$
of the regular domain $\Omega$, \cite[Remark 7.1.5]{bu}, and we are entitled to repeat the
proof of \cite[Theorem 7.1.7]{bu}.
As a concrete example of such sets, we observe that for all $x\in\X$ and for almost-every $\rho>0$, any ball $B_\rho(x)$ is a regular domain, \cite{bu,bcm}.
\end{enumerate}
We refer also to \cite[Section 4]{bcm} for refined versions of the results of \cite{bu} in terms of essentially bounded \textit{divergence-measure vector fields} and of an alternative notion of $BV$ functions defined via the differential machinery of \cite{gi}.
\end{remark}

\bigskip

\section{Trace comparison}\label{comparison}

\medskip

In this last section we compare the foregoing discussion on the rough trace with \cite{LS}, where the authors investigate the properties of the trace operator for $BV$ functions by means of the more classical Lebesgue-points characterization.


We start by summarizing the salient definitions and results of \cite{LS} which will be of relevance to us.

\medskip

\begin{definition}\label{trace}Let $\Omega\subset{\X}$
be an open set and let $u$ be a $\mu$-measurable function on $\Omega$.
Then, we shall say that a function $\mathrm{T}u:\partial\Omega\rightarrow\mathbb{R}$
is a \textit{\textcolor{black}{trace }}of $u$ if for ${\cal S}^{h}$-almost
every $x\in\partial\Omega$ one has
\[
\underset{\rho\rightarrow0^{+}}{\lim}\mean{\Omega\cap B_{\rho}(x)}\left|u-\mathrm{T}u(x)\right|\mathrm{d}\mu=0.
\]
\end{definition}


The \textit{zero extension} of a measure $\mu$ from $\Omega$ to $\overline{\Omega}$, written as
$\bar{\mu}$, is given by $\bar{\mu}(A)\coloneqq\mu(A\cap\Omega)$
whenever $A\subset\overline{\Omega}$; in a similar fashion, we shall write $\overline{\mathcal{S}}^h$ to intend the spherical Hausdorff measure on $\partial\Omega$ corresponding to the measure $\bar{\mu}$ on $\overline{\Omega}$.

Accordingly, for any measurable function $u$ in $\Omega$, its zero-extension to $\overline{\Omega}$ will be written as $\bar{u}$; $\bar{u}^\vee$ and $\bar{u}^\wedge$ will therefore denote the approximate limits of $\bar{u}$ computed in terms of the extended measure $\bar{\mu}$.

\medskip

\begin{proposition}\label{du-zero}
{\normalfont{\cite[Proposition 3.3]{LS}}}
Let $\Omega\subset{\X}$
be a bounded open set supporting a $(1,1)$-Poincar\'e inequality and
assume that $\mu$ is doubling on $\Omega$. Let $\overline{\Omega}$
be equipped with the extended measure $\bar{\mu}$. If $u\in BV(\Omega)$,
then its zero-extension $\bar{u}$ to
$\overline{\Omega}$ is such that $\left\Vert \bar{u}\right\Vert _{BV\left(\overline{\Omega}\right)}=\left\Vert u\right\Vert _{BV(\Omega)}$,
whence
\begin{equation*}
 \left\Vert D\bar{u}\right\Vert (\partial\Omega)=0.
\end{equation*}
\end{proposition}

\smallskip

\begin{definition}
We say that an open set $\Omega$
satisfies a \textit{\textcolor{black}{measure-density condition }}\textcolor{black}{if
there exists a constant $C>0$ such that}
\begin{equation}\label{meas-dens}
\mu\left(B_{\rho}(x)\cap\Omega\right)\ge C\mu\left(B_{\rho}(x)\right)
\end{equation}
for ${\cal S}^{h}$-almost every $x\in\partial\Omega$ and for every
$\rho\in(0,\mathrm{diam}(\Omega))$.
\end{definition}

\smallskip

\begin{thm}\label{tracethm}{\normalfont{\cite[Theorem 3.4]{LS}}}
Let $\Omega\subset\X$ be a bounded open set that supports a $(1,1)$-Poincar\'e inequality, and assume that $\mu$ is doubling on $\Omega$. Then, there exist $q>1$ depending only on the doubling constant in $\Omega$ and a linear trace operator
$\mathrm{T}$ on $BV(\Omega)$ such that, given $u\in 
BV(\Omega)$,
for $\overline{{\cal S}}^{h}$-almost every $x\in\partial
\Omega$ we have
\begin{equation}\label{trace-leb}
\underset{\rho\rightarrow0^{+}}{\lim}\mean{\Omega\cap 
B_{\rho}(x)}\left|u-\mathrm{T}u(x)\right|^{\frac{q}{q-1}}
\mathrm{d}\mu=0.
\end{equation}
If $\Omega$ also satisfies the measure-density condition \eqref{meas-dens}, 
the above holds for ${\cal S}^{h}$-almost
every $x\in\partial\Omega$.
\end{thm}

\smallskip

\begin{remark}\label{eq-aplim}
In the proof of \cite[Theorem 3.4]{LS}, the authors used the condition $\|D\bar{u}\|(\partial\Omega)$ $=0$ found in   Proposition \cite[Proposition 3.3]{LS} to infer  that $\overline{\mathcal{S}}^h\!\left(S_{\bar{u}}\cap\partial\Omega\right)$ $=0$; this of course arises from the decomposition of the total variation measure given in Remark \ref{bv-decomp} and entails that the equality
\[
\bar{u}^{\wedge}(x)=\bar{u}^{\vee}(x)
\]
holds for $\overline{{\cal S}}^{h}$-almost every $x\in\partial\Omega$.

Always in the proof of \cite[Theorem 3.4]{LS},  \eqref{trace-leb} was actually found to hold in the form
\[
 \underset{\rho\to 0^+}{\lim}\mean{\Omega\cap B_\rho(x)} |u-\bar{u}^\wedge(x)|^{\frac{q}{q-1}}\text{d}\mu=0.
\]
Then, one sets $\mathrm{T}u(x)=\bar{u}^\wedge(x)$, which in turn equals $\bar{u}^\vee(x)$ for $\overline{\mathcal{S}}^h$-almost every $x\in\partial\Omega$ by the above considerations.
In particular, if $\Omega$ satisfies the measure-density condition \eqref{meas-dens}, these latter equalities are fulfilled for $\mathcal{S}^h$-almost every $x\in\partial\Omega$ as well.
\end{remark}

Next, we prove that the rough trace of a $BV$ function is bounded by the approximate limits of its zero-extension to $\overline{\Omega}$, and that $\text{T}u=u^*$ on $\partial\Omega$.


\begin{thm}
\label{u-star-aplim}
Let $\Omega\subset\X$
be an open set and let $u\in BV(\Omega)$. 
Then, for every $x\in\partial^{*}\Omega$  such that
$u^*(x)>-\infty$, we have  that
\[
\bar{u}^{\wedge}(x)\le u^{*}(x)\le\bar{u}^{\vee}(x).
\]
In particular, if $\Omega$ is a bounded open set supporting a (1,1)-Poincar\'e inequality and $\mu$ is doubling on $\Omega$, then
\[
 \mathrm{T}u(x)=\bar{u}^\wedge(x)=\bar{u}^\vee(x)=u^*(x)
\]
for $\overline{\mathcal{S}}^h$-almost every $x\in\partial\Omega$.

If in addition the measure-density condition \eqref{meas-dens} is satisfied, then the above equality holds $\mathcal{S}^h$-almost everywhere on $\partial\Omega$.

Finally, assume also that $\|D\one_\Omega\|(\X)<\infty$ and $\mathcal{S}^h(\partial\Omega\backslash\partial^*\Omega)=0$. Then, one has
\[
 \|\mathrm{T}u\|_{L^1(\partial\Omega,\mathcal{S}^h)}\le c\|u\|_{BV(\Omega)}
\]
with a constant $c>0$ independent of $u$, if and only if there exists $\delta>0$ such that for every $E\subset\Omega$ with $\text{diam}(E)\le\delta$ and $\|D\one_E\|(\Omega)<\infty$ there holds
\[
 \|D\one_E\|(\Omega^c)\le c'\|D\one_E\|(\Omega)
\]
with a constant $c'>0$ independent of $E$. 
\end{thm}
\begin{proof}
Recall that, by Definition \ref{rough-def}, $u^{*}(x)$ is the supremum of those $t\in\R$ for which $\|D\one_{E_t}\|(\X)<\infty$ and $x\in\partial^{*}E_{t}$, which explains the requirement $u^*(x)>-\infty$ in our statement. 

We have that
\[
\bar{u}^\vee(x)=\inf \left\{ t\in \R: \lim_{\rho\to 0}
\frac{\bar\mu(\{\bar  u>t\} \cap B_\rho(x))}{\bar\mu(B_\rho(x))}
=0\right\}.
\]
Here, the balls have to be understood as balls on the metric space $\overline{\Omega}$;
then, from the definition of $\bar\mu$ we get
\begin{align*}
\frac{\bar\mu(\{\bar  u>t\} \cap B_\rho(x))}{\bar\mu(B_\rho(x))}
=&
\frac{\bar\mu(\{\bar  u>t\} \cap B_\rho(x)
\cap \overline {\Omega})}{\bar\mu(B_\rho(x)\cap \overline{\Omega})}\\
=&
\frac{\mu(\{u>t\} \cap B_\rho(x))}{\mu(B_\rho(x)\cap \Omega)},
\end{align*}
where we have also taken into account that
$\{u>t\}\subset \Omega$.

If $\bar{u}^\vee(x)=+\infty$ there is nothing to prove;
otherwise, if $t>\bar{u}^\vee(x)$ and 
$x\in \partial^*\Omega$ we obtain that
\begin{align*}
\limsup_{\rho\to 0}
\frac{\mu(\{u> t\}\cap B_\rho(x))}{\mu(B_\rho(x))}
=&
\limsup_{\rho\to 0}
\frac{\mu(\{u> t\}\cap B_\rho(x))}{\mu(B_\rho(x)\cap  \Omega)} \cdot 
\frac{\mu(B_\rho(x)\cap  \Omega)}{\mu(B_\rho(x))}\\
\leq&
\lim_{\rho\to 0}
\frac{\mu(\{u> t\}\cap B_\rho(x))}{\mu(B_\rho(x)\cap  \Omega)}
=0.
\end{align*}
As  a consequence, noticing that for $t>s>\bar{u}^\vee(x)$,
$\{ u>t\} \subseteq\{ u\geq t\}\subseteq \{u>s\}$, we deduce
$x\in E_t^{(0)}$ for any $t>\bar{u}^\vee(x)$  and then
$x\not\in \partial^*E_t$. Hence,
$u^*(x)\leq t$ for  any $t>\bar{u}^\vee(x)$, and so
$u^*(x)\leq \bar{u}^\vee(x)$.

In the same way, if $\bar{u}^\wedge(x)=-\infty$
there is nothing to prove, otherwise
if $t<\bar{u}^\wedge(x)$ and
$x\in \partial^*\Omega$, we get that
\[
\limsup_{\rho\to 0}
\frac{\mu(\{u< t\}\cap B_\rho(x))}{\mu(B_\rho(x))}=0,
\]
and then $\Theta^*_\mu(\{u<t\},x)=0$ and  then
$x\in E_t^{(1)}$,  i.e.  $x\not\in \partial^*E_t$ 
for any $t<\bar{u}^\wedge(x)$. 
Since $u^*(x)>-\infty$, we have that $x\in \partial^*E_t$ for some  $t\in \R$, and the previous computation implies that
$t\geq \bar{u}^\wedge(x)$.  So we can conclude that
$\bar{u}^\wedge(x)\leq u^*(x)$.

The other assertions in the Theorem follow from 
Theorem \ref{mazyathm}, Proposition \ref{du-zero}
and Theorem \ref{tracethm}.

\end{proof}

\smallskip

\begin{remark}\textbf{(Comments and Open Problems)}
In conclusion, our discussion allowed us to find the conditions to impose on a domain $\Omega\subset\X$ in order to ensure that the ``classical'' trace $\text{T}u$ and the rough trace $u^*$ of a $BV$ function coincide $\mathcal{S}^h$-almost everywhere on the boundary of such domain.

Actually, our results also address the $L^1$-summability of the trace $\text{T}u$; indeed, as we can see from Theorem \ref{u-star-aplim}, if we introduce the additional assumption that for some $\delta>0$ and for any set $E\subset\Omega$ with finite perimeter in $\Omega$ and $\text{diam}(E)\le\delta$ it holds
\[
 \|D\one_E\|(\Omega^c)\le c\|D\one_E\|(\Omega)
\]
for some constant $c>0$ independent of $E$, which is namely the fundamental condition \eqref{per-compl} of Theorem \ref{mazyathm}, then we get that $\text{T}u\in L^1(\partial\Omega,\mathcal{S}^h)$ as well.

In \cite[Section 5]{LS}, the authors tackle the issue of the summability of the trace $\text{T}u$ by working again in terms of the measure-density condition \eqref{meas-dens}  and assuming an additional ``surface-density'' condition for $\partial\Omega$, namely that there is a constant $c=c_{\partial\Omega}>0$ such that
\[
 \mathcal{S}^h(B_\rho(x)\cap\partial\Omega)\le c\frac{\mu(B_\rho(x))}{\rho}
\]
 for any $x\in\partial\Omega$ and any $\rho\in (0,2\text{diam}(\Omega))$.
 
 Thus said, one question arises naturally: \textit{how does the requirement \eqref{per-compl} in Theorem \ref{mazyathm} relate with the measure-density condition \eqref{meas-dens} and with the surface-density condition above?}
 
 Answering to such a question would be of general interest as it would provide us with a better understanding of the domains where the ``nice'' properties of traces of $BV$ functions are satisfied, and therefore we would have a more consistent and more comprehensive theory of traces of $BV$ functions.
\end{remark}



\end{document}